\documentclass{amsart}

\usepackage{amssymb}
\usepackage{graphicx}
\usepackage{amsthm}
\usepackage{parskip}
\usepackage{float}
\usepackage{tikz}

\newcommand\ZZ{\mathbb{Z}}
\newcommand\NN{\mathbb{N}}
\newcommand\CC{\mathcal{C}}

\newcommand\QQ{\mathbb{Q}}
\newcommand\QQP{\mathbb{Q}(\sqrt{-p})}

\newcommand\Anz{A_{2^{n}}}
\newcommand\Ano{A_{2^{n-1}}}
\newcommand\Ant{A_{2^{n-2}}}
\newcommand\Hnz{H_{2^{n}}}
\newcommand\Hno{H_{2^{n-1}}}
\newcommand\Hnt{H_{2^{n-2}}}

\newcommand\DD{\mathcal{D}}

\newcommand\ow{\overline{w}}
\newcommand\oz{\overline{z}}

\newcommand\mm{\mathfrak{m}}
\newcommand\pp{\mathfrak{p}}

\newcommand\pt{\mathfrak{t}}
\newcommand\zz{\mathfrak{Z}}
\newcommand\aaa{\mathfrak{a}}

\newcommand\Gal{\text{Gal}}

\newcommand\Imag{\text{Im}}
\newcommand\Real{\text{Re}}

\theoremstyle{theorem} \newtheorem{theorem}{Theorem}
\theoremstyle{theorem} \newtheorem{lemma}{Lemma}
\theoremstyle{theorem} \newtheorem{corollary}{Corollary}
\theoremstyle{theorem} \newtheorem{conjecture}{Conjecture}
\theoremstyle{theorem} \newtheorem*{hypoR}{Hypothesis (R)}
\theoremstyle{theorem} \newtheorem*{hypoB}{Hypothesis (B)}
\theoremstyle{theorem}\newtheorem{prop}{Proposition}
\theoremstyle{remark} \newtheorem*{remark}{Remark}
\theoremstyle{remark} 
\theoremstyle{remark} 
\theoremstyle{remark} \newtheorem*{example}{Example}

\DeclareMathOperator*{\sumsum}{\sum\sum}

\title{The infinitude of $\QQP$ with class number divisible by $16$}
\author{Djordjo Milovic}
\address{D{\a'e}partement de Math{\a'e}matiques \newline B{\a^a}timent 425 \newline Facult{\a'e} des Sciences d'Orsay \newline Universit{\a'e} Paris-Sud \newline F-91405 Orsay Cedex \newline France
 \newline\newline Mathematisch Instituut\newline Universiteit Leiden \newline Niels Bohrweg 1\newline 2333 CA Leiden \newline The Netherlands}

\begin{document}
\maketitle
\let\thefootnote\relax\footnote{Partially supported by an ALGANT Erasmus Mundus Scholarship.}

\begin{abstract}
The density of primes $p$ such that the class number $h$ of $\QQP$ is divisible by $2^k$ is conjectured to be $2^{-k}$ for all positive integers $k$. The conjecture is true for $1\leq k\leq 3$ but still open for $k\geq 4$. For primes $p$ of the form $p = a^2 + c^4$ with $c$ even, we describe the 8-Hilbert class field of $\QQP$ in terms of $a$ and $c$. We then adapt a theorem of Friedlander and Iwaniec to show that there are infinitely many primes $p$ for which $h$ is divisible by $16$, and also infinitely many primes $p$ for which $h$ is divisible by $8$ but not by $16$.
\end{abstract}

\section{Introduction}
Let $p$ be a prime number, and let $\CC$ and $h$ be the class group and the class number of $\QQP$, respectively. Since the discriminant of this field is either $-p$ or $-4p$, Gauss's genus theory implies that the $2$-part of $\CC$ is cyclic, and so the structure of the $2$-part of the class group is entirely determined by the highest power of $2$ dividing $h$. More precisely, Gauss's genus theory implies that   
$$2|h \Longleftrightarrow p \equiv 1 \bmod{4}.$$
The criterion
$$4|h \Longleftrightarrow p \equiv 1 \bmod{8}$$
can be deduced easily from R\'{e}dei's work on the $4$-rank of quadratic number fields \cite{Redei}. In \cite{BarCohn}, Barrucand and Cohn gave an explicit criterion for divisibility by $8$ by successively extracting square roots of the class of order two. It states that
$$8|h\Longleftrightarrow p = x^2 + 32y^2\text{ for some integers } x \text{ and }y.$$
This can be restated as
\begin{equation}\label{div8}
8|h\Longleftrightarrow p \equiv 1 \bmod{8} \text{ and } 1+i \text{ is a square modulo }p
\end{equation}
where $i$ is a square root of $-1$ modulo $p$ (see \cite[(10), p.68]{BarCohn}). In \cite{Ste1}, Stevenhagen also obtained the criterion \eqref{div8}, albeit by a more abstract argument using class field theory over the field $\QQ(i)$.
\\\\
Given a subset $S$ of the prime numbers, and a  real number $X\geq 2$, define 
$$R(S, X) := \frac{\#\{p\leq X\text{ prime }: p\in S\}}{\#\{p\leq X\text{ prime }\}}.$$
If the limit $\lim_{X\rightarrow\infty}R(S, X)$ exists, we denote it by $\rho(S)$ and call it the \textit{natural density of} $S$. Let
$$S(n) = \left\{p\text{ prime }:\ n|h(-p) \right\};$$
here we write $h(-p)$ for the class number of $\QQP$ to emphasize its dependence on $p$. From the above, it is clear that $\rho(S(2)) = 1/2$ and $\rho(S(4)) = 1/4$. From \eqref{div8}, we see that $8$ divides $h$ if and only if $p$ splits completely in $\QQ(\zeta_8, \sqrt{1+i})$, where $\zeta_8$ is a primitive $8^{\text{th}}$ root of unity. Since this is a degree $8$ extension of $\QQ$, \v{C}ebotarev's density theorem implies that $\rho(S(8)) = 1/8$. For a discussion of these and similar density results, see \cite[p.16-19]{Ste2}.
\\\\
The Cohen-Lenstra heuristics \cite{CohenLenstra} can be adapted to this situation to predict the density of primes $p$ such that $2^k$ divides $h$ for $k\geq 1$. Cohen and Lenstra stipulate that an abelian group $G$ occurs as the class group of an imaginary quadratic field with probability proportional to the inverse of the size of the automorphism group of $G$. Under this assumption, the cyclic group of order $2^{k-1}$ would occur as the $2$-part of the class group of an imaginary quadratic number field twice as often as the cyclic group of order $2^k$. As we just saw above, $\rho(S(2^k)) = \frac{1}{2}\rho(S(2^{k-1}))$ for $k\leq 3$, so we are led to conjecture
\begin{conjecture}
For all $k\geq 1$, $\lim_{X\rightarrow \infty}R(S(2^k), X)$ exists and is equal to $2^{-k}$.
\end{conjecture}
While Conjecture 1 is true for $k\leq 3$, it has not been proven for any $k\geq 4$, and a proof along the lines of the arguments for $k\leq 3$ seems out of reach (see \cite[p. 16]{Ste2}). Although several criteria for divisibility by $16$ have been found already (see \cite{Kaplan77}, \cite{W81}, and \cite{LW82}), none of them appear to be sufficient to produce even infinitely many primes $p$ for which the class number of $\QQP$ is divisible by $16$. This is precisely our aim in this paper -- we will show that there is an infinite number of primes $p$ for which $16|h$ and also an infinite number of primes $p$ for which $8|h$ but $16\nmid h$. We also derive some consequences for the fundamental unit $\epsilon_p$ of the real quadratic number field $\QQ(\sqrt{p})$. 
\\\\
We tackle the question of infinitude not by developing a new criterion for divisibility by $16$ which handles all primes, but by focusing on a very special subset of primes. These are the primes of the form 
\begin{equation}\label{pcForm}
p = a^2 + c^4,\ \ \ \ c \text{  even}.
\end{equation}
The main theorem that we prove gives a new and very explicit criterion for divisibility by $16$ of class numbers of $\QQP$ for $p$ of the form \eqref{pcForm}.
\begin{theorem}
Suppose $p$ is a prime of the form $a^2+c^4$, where $a$ and $c$ are integers. Let $h_{-4p}$ denote the class number of $\QQ(\sqrt{-p})$.
\\\\
(i) If $a \equiv \pm 1 \bmod{16}$ and $c \equiv 0 \bmod{4}$, then $h_{-4p} \equiv 0 \bmod{16}$.
\\\\
(ii) If $a \equiv \pm 3 \bmod{16}$ and $c \equiv 2 \bmod{4}$, then $h_{-4p} \equiv 0 \bmod{16}$.
\\\\
(iii) If $a \equiv \pm 7 \bmod{16}$ and $c \equiv 0 \bmod{4}$, then $h_{-4p} \equiv 8 \bmod{16}$.
\\\\
(iv) If $a \equiv \pm 5 \bmod{16}$ and $c \equiv 2 \bmod{4}$, then $h_{-4p} \equiv 8 \bmod{16}$.
\end{theorem}
Once we prove Theorem 1, the infinitude of primes $p$ of the form as in the statements $(i)-(iv)$ of the theorem follows from the following generalization of a powerful theorem of Friedlander and Iwaniec (see \cite[Theorem 1]{FI1}):
\begin{prop}
Let $a_0\in\{1, 3, 5, 7, 9, 11, 13, 15\}$ and $c_0\in\{0, 2\}$. Then, uniformly for $X\geq 3$, we have the equality
\begin{equation}\label{prop1eq}
\sumsum_{\substack{a^2+c^4\leq X \\ a\equiv a_0\bmod{16} \\ c\equiv c_0\bmod{4} \\ a^2+c^4\text{ prime}}}1 = \frac{\kappa}{2\pi} \frac{X^{3/4}}{\log X}\left(1+O\left(\frac{\log\log X}{\log X}\right)\right),\end{equation}
where $a$ and $c$ run over all integers and 
$$\kappa = \int_0^1(1-t^4)^{\frac{1}{2}}dt\approx 0.874\ldots.$$
In particular, there exist infinitely many primes of the form $a^2+c^4$ with $a\equiv a_0\bmod{16}$ and $c\equiv c_0\bmod{4}$.
\end{prop}
Theorem 1 and Proposition 1 immediately imply:
\begin{corollary}\label{COR1}
For a prime $p$, let $h_{-4p}$ denote the class number of $\QQ(\sqrt{-p})$.
Then, for sufficiently large $X$, we have 
$$\#\{p\leq X: h_{-4p}\equiv 0\bmod{16}\}\geq \frac{X^{3/4}}{4\log X}$$
and 
$$\#\{p\leq X: h_{-4p}\equiv 8\bmod{16}\}\geq \frac{X^{3/4}}{4\log X}.$$
\end{corollary}
The proof of Proposition 1 will take a significant portion of our paper. Although the ideas required to generalize \cite[Theorem 1]{FI1} in this way are not particularly deep, implementing them turns out to be quite complicated simply because the proof of \cite[Theorem 1]{FI1} itself is very difficult. One can thus view Sections 4-6 as a summary of the proof of \cite[Theorem 1]{FI1} in a slightly more general context.
\\\\
Since primes of the form $a^2+c^4$ with $c$ even have density $0$ in the set of all primes, our methods cannot be used to tackle Conjecture 1. Nonetheless, each of the cases $(i)-(iv)$ in Theorem 1 occurs with the same density among all primes this form, so the conjecture for $k = 4$ deduced from the Cohen-Lenstra heuristics above holds within the thin family of imaginary quadratic number fields $\QQP$ where $p$ is a prime of the form $a^2+c^4$ with $c$ even. This is yet another piece of evidence suggesting that Conjecture 1 is true for $k = 4$. However, we also note that Conjecture 1 for $k = 4$ does not imply Corollary \ref{COR1}, as knowledge of the behavior of the class numbers of $\QQP$ over the set of all primes $p$ does not necessarily give information about their behavior over a thin subset of all primes. 
\\\\
We now give a consequence of our results and a criterion for divisibility by $16$ due to Williams \cite{W81}. Let $p\equiv 1\bmod 8$, and let $\epsilon_p$ be a fundamental unit of the real quadratic field $\QQ(\sqrt{p})$, written in the form $\epsilon_p = T + U\sqrt{p}$, where $T$ and $U$ are integers. The criterion states that if $8|h$, then 
\begin{equation}\label{critT}
h\equiv T + p - 1 \bmod{16},
\end{equation}
so that $16|h$ if and only if $T\equiv 1-p\bmod{16}$. An immediate byproduct of Theorem 1 and criterion \eqref{critT} is the following corollary.
\begin{corollary}\label{CORF1}
Suppose $p$ is a prime of the form $a^2+c^4$, where $a$ is odd and $c$ is even. Let $\epsilon_p = T + U\sqrt{p}$ denote a fundamental unit of $\QQ(\sqrt{p})$.
\\\\
(i) If $a \equiv \pm 1 \bmod{16}$ and $c \equiv 0 \bmod{4}$, then $T \equiv 0 \bmod{16}$ and $U\equiv \pm 1\bmod 8$.
\\\\
(ii) If $a \equiv \pm 3 \bmod{16}$ and $c \equiv 2 \bmod{4}$, then $T \equiv 8 \bmod{16}$ and $U\equiv \pm 5\bmod 8$.
\\\\
(iii) If $a \equiv \pm 7 \bmod{16}$ and $c \equiv 0 \bmod{4}$, then $T \equiv 8 \bmod{16}$ and $U\equiv  \pm 1\bmod 8$.
\\\\
(iv) If $a \equiv \pm 5 \bmod{16}$ and $c \equiv 2 \bmod{4}$, then $T \equiv 0 \bmod{16}$ and $U\equiv \pm 5\bmod 8$.
\end{corollary}
This can be viewed as an extension of \cite[Corollary 1.2(i), p.115-116]{Lag} to primes of the form $p = a^2+c^4$. Now Proposition 1 gives
\begin{corollary}\label{CORF2}
For a prime $p\equiv 1\bmod 8$, let $\epsilon_p = T + U\sqrt{p}$ denote the fundamental unit of $\QQ(\sqrt{p})$. Then, for sufficiently large $X$, we have 
$$\#\{p\leq X: p\equiv 1\bmod 8,\ T\equiv 0\bmod{16}\}\geq \frac{X^{3/4}}{4\log X}$$
and 
$$\#\{p\leq X: p\equiv 1\bmod 8,\ T \equiv 8\bmod{16}\}\geq \frac{X^{3/4}}{4\log X}.$$
\end{corollary}
The existence of infinitely many $p\equiv 1\bmod 8$ such that $T\equiv T_0 \bmod 16$ for a fixed $T_0\in\{0, 8\}$ is not at all trivial. Hence Corollary \ref{CORF1} sheds some new light on the fundamental unit $\epsilon_p$ of $\QQ(\sqrt{p})$, one of the most mysterious quantities in number theory.

\section{Hilbert class fields} 
Suppose $p\equiv 1\pmod 4$. Then there are two finite primes of $\QQ$ which ramify in $\QQP$, namely $2$ and $p$. The prime $\pp = (\sqrt{-p})$ of $\QQP$ lying above $p$ is principal, and so its ideal class in $\CC$ is the identity. Genus theory then implies that the class of the prime ideal $\pt = (2, 1+\sqrt{-p})$ of $\QQP$ lying above $2$ is the unique element of order two in $\CC$. Assuming that $h$ is divisible by $2^n$ for some non-negative integer $n$, to check that it is divisible by $2^{n+1}$, it would suffice to check that the class of $\pt$ belongs to $\CC^{2^n}$.

\subsection{$2^n$-Hilbert class fields}
Recall that the \textit{Hilbert class field} $H$ of $K = \QQP$ is the maximal unramified abelian extension of $\QQP$. The Artin symbol induces a canonical isomorphism of groups
\begin{equation}\label{ArtSym}
\left(\frac{\cdot}{H/K}\right):\CC\longrightarrow \Gal(H/\QQP). 
\end{equation}
Suppose for the moment that $2^n|h$ for some non-negative integer $n$. Then $\CC^{2^n}$ is a subgroup of $\CC$ of index $2^n$. We define the $2^n$-\textit{Hilbert class field} $H_{2^n}$ to be the subfield of $H$ fixed by the the image of $\CC^{2^n}$ under the isomorphism \eqref{ArtSym}. Since the $2$-primary part of $\CC$ is cyclic, it follows immediately that $H_{2^n}$ is the unique unramified, cyclic, degree-$2^n$ extension of $K$. Moreover, \eqref{ArtSym} induces a canonical isomorphism of cyclic groups of order $2^n$
\begin{equation}\label{ArtSym2n}
\left(\frac{\cdot}{H_{2^n}/K}\right):\CC/\CC^{2^n}\longrightarrow \Gal(H_{2^n}/K). 
\end{equation}
Hence $\pt$ belongs to $\CC^{2^n}$ if and only if $\pt$ has trivial Artin symbol in $\Gal(H_{2^n}/K)$. By class field theory, this is equivalent to $\pt$ splitting completely in $H_{2^n}$.
\\\\
The main idea of the proof of Theorem 1 is to write down explicitly the $8$-Hilbert class field $H_8$ of $\QQP$, and then to characterize those $p$ such that $\pt$ splits completely in $H_8$. We remark here that although Cohn and Cooke \cite{CohnCooke} have already written down $H_8$ in terms of the fundamental unit $\epsilon_p$ of the real quadratic number field $\QQ(\sqrt{p})$ and certain integer solutions $u$ and $v$ to $p = 2u^2-v^2$, not enough is known about either $\epsilon_p$ or $u$ and $v$ to deduce anything about the distribution of primes $p$ such that $\pt$ splits completely in $H_8$.

\subsection{Generating $2^n$-Hilbert class fields}
We first state and prove some lemmas which will prove to be useful in our quest to explicitly generate $H_8$.
\\\\
The $2$-Hilbert class field, also called the \textit{genus field} of $\QQP$, is known to be $H_2 = \QQ(i, \sqrt{p})$. Hence every $2^n$-Hilbert class field of $\QQP$ contains $\QQ(i)$, and so we can study the splitting behavior of $\pt$ in $H_{2^n}$ by working over the quadratic subfield $\QQ(i)$ of $H_2$. With this in mind, we now state some well-known generalities about the completion of $\QQ(i)$ with respect to the prime ideal $(1+i)$ lying over $2$.
\\\\
This completion is $\QQ_2(i)$, and its ring of integers $\ZZ_2[i]$ is a discrete valuation ring with maximal ideal $\mm$ and uniformizer $m = 1+i$. Let $U = (\ZZ_2[i])^{\times}$ denote the group of units of $\ZZ_2[i]$ and for each positive integer $k$, define $U^{(k)} = 1+\mm^k$. Then there is a filtration
$$U = U^{(1)}\supset U^{(2)}\supset\cdots\supset U^{(k)}\supset\cdots.$$
For any $k\geq 3$, squaring gives an isomorphism $U^{(k)}\xrightarrow{\sim} U^{(k+2)}$. Indeed, let $1+m^{k+2}y \in U^{(k+2)}$. Hensel's lemma implies that there exists $x\in \mm^{k-2}$ such that $x^2+x = m^{k-2}y$. Then $(1+2x)^2 = 1+m^{k+2}y$ and $1+2x \in U^{(k)}$. It is not hard to see that 
$$U = \left\langle i\right\rangle\times U^{(3)} = \left\langle i\right\rangle\times\left\langle 2+i\right\rangle\times U^{(4)},$$
so that $U^2 = \left\langle -1\right\rangle \times U^{(5)}$. In other words, $u\in U$ is a square in $\QQ_2(i)$ if and only if $u\equiv \pm 1\pmod{\mm^5}$. Moreover, if $\omega\equiv \pm 1\pmod \mm^4$ (or if $\omega\neq \pm 1\pmod \mm^4$), then the minimal polynomial of $\sqrt{\omega}$ over $\ZZ_2[i]$ reduces to $X^2+X$ or $X^2+X+1$ (respectively $X^2+1$) modulo $\mm$. We collect these observations into the following lemma.
\begin{lemma}\label{ptunram}
Let $\omega$ be a unit in $\ZZ_2[i]$. Then $\QQ_2(i, \sqrt{\omega})$ is unramified over $\QQ_2(i)$ if and only if $\omega\equiv \pm 1\pmod{\mm^4}$. Moreover, $\QQ_2(i, \sqrt{\omega})=\QQ_2(i)$, i.e. $\omega$ is a square in $\QQ(i)$ if and only if $\omega\equiv \pm 1\pmod \mm^5$. 
\end{lemma}
Next, we state two lemmas which we will use to check that the extensions of $\QQP$ which we construct are normal and cyclic.
\begin{lemma}\label{lemNor}
Let $K$ be a field of characteristic different from $2$, let $d$ be an element of $K$ which is not a square in $K$, and let $L = K(\sqrt{d})$. Let $a, b\in K$ such that $a+b\sqrt{d}$ is not a square in $L$ and let $M = L(\sqrt{a+b\sqrt{d}})$. Then $L/K$ is cyclic of degree $4$ if and only if $a^2-db^2\in d\cdot K^2$.
\end{lemma}
\begin{proof}
See \cite[Chapter VI, Exercise 4, p.321]{Lang}.
\end{proof}
\begin{lemma}\label{lemCyclic}
Let $K$ be a field. Suppose $M/K$ is a cyclic extension of degree $2m$ and let $\sigma$ be a generator of $\Gal(M/K)$. Let $L$ be the subfield of $M$ fixed by $\sigma^m$. Suppose $N/K$ is a Galois extension containing $M$ such that $N/L$ is cyclic of degree $4$. Then $N/K$ is cyclic of degree $4m$.     
\end{lemma}
\begin{proof}
Let $\sigma_1$ denote a lifts of $\sigma$ to $\Gal(N/K)$. The order of $\sigma_1$ is at least $2m$ since the order of $\sigma$ is $2m$. As $\sigma^m$ fixes $L$, $\sigma_1^m$ is an element of $\Gal(N/L)$which is non-trivial on $M$ and hence has order $4$. Thus the order of $\sigma_1$ is $4m$.
\end{proof}
Finally, we arrive at the main lemma we will use to construct $2^n$-Hilbert class fields from $2^{n-1}$-Hilbert class fields. This result is inspired by a theorem of Reichardt \cite[3. Satz, p.82]{Reichardt}. His theorem proves the existence of generators $\sqrt{\varpi}$ for $H_{2^n}$ over $H_{2^{n-1}}$ with $\varpi\in H_{2^{n-1}}$ of a certain form. We prove sufficient conditions for an element $\varpi$ of a similar form to give rise to a generator, so that we can actually construct $H_{2^n}$. 
\begin{lemma}\label{lemH2n}
Let $h$ be the class number of $\QQP$, let $n\geq 2$, and suppose that $2^n$ divides $h$. Suppose $\Ano$ is a degree $2^{n-1}$ extension of $\QQ$ such that:
\begin{itemize}
  \item $\Ano\subset \Hno$,
	\item $\Ano$ contains $\QQ(i)$ and $(1+i)$ is unramified in $\Ano/\QQ(i)$, and
	\item there is a prime element $\varpi$ in $\Ano$ such that:
	\begin{itemize}
		\item $\varpi$ lies above $p$ and its ramification and inertia indices over $p$ are equal to $1$,
		\item denoting the conjugate of $\varpi$ over $\Ant = \Ano\cap\Hnt$ by $\varpi'$, we have $\Hno = \Hnt(\sqrt{\varpi\varpi'}) = \Ano(\sqrt{\varpi\varpi'})$,
		\item $(U_2)$: $(1+i)$ remains unramified in $\Anz = \Ano(\sqrt{\varpi})$, and
		\item $(N)$: $\Hno(\sqrt{\varpi})$ is normal over $\QQ$.
	\end{itemize}
\end{itemize}
Then $\Hnz = \Hno(\sqrt{\varpi})$. 
\end{lemma}
\begin{proof}
Since the ramification index of $\varpi$ over $p$ is $1$, $\varpi$ and $\varpi'$ are coprime in $\Ano$.
\\\\
First we check that $\varpi$ is not a square in $\Hno$. Note that $\Hno$ is normal over $\Ant$, while $\Anz$ is not normal over $\Ant$ ($\sqrt{\varpi'}$ is not an element of $\Anz$). Hence $\Anz$ cannot be a subfield of $\Hno$ and so $\sqrt{\varpi}\notin\Hno$.
\\\\
By assumption, $\Hno(\sqrt{\varpi})$ is normal over $\QQ$, and hence also over $\QQP$ and $\Hnt$. Since $\varpi$ and $\varpi'$ are conjugates over $\Ant$, they are also conjugates over $\Hnt$. As $\Hno = \Hnt(\sqrt{\varpi\varpi'})$ and $\varpi\varpi' = \varpi\varpi'\cdot 1^2$, Lemma \ref{lemNor} implies that $\Hno(\sqrt{\varpi})$ is degree $4$ cyclic extension of $\Hnt$. Moreover, $\Hno$ is a degree $2^{n-1}$ cyclic extension of $\QQP$, so Lemma \ref{lemCyclic} implies that $\Hno(\sqrt{\varpi})$ is a degree $2^n$ cyclic extension of $\QQP$.
\\\\
It remains to show that $\Hno(\sqrt{\varpi})/\QQP$ is unramified. We will establish this by showing that each of the ramification indices of the primes $2$ and $p$ in $\Hno(\sqrt{\varpi})$ is at most $2$. 
\\\\
The prime $2$ ramifies in $\QQ(i)$, but by assumption $(1+i)$ is unramified in $\Anz$. As $\Hno(\sqrt{\varpi}) = \Anz(\sqrt{\varpi\varpi'})$ and $p\equiv 1\bmod 4$, Lemma \ref{ptunram} ensures that $(1+i)$ is unramified in $\Hno(\sqrt{\varpi})$. Hence the ramification index of $2$ in $\Hno(\sqrt{\varpi})$ is $2$.
\\\\
Now note that $\Anz$ is a subfield of $\Hno(\sqrt{\varpi})$ of index $2$. The ramification index of the prime $\varpi'$ over $p$ is $1$. Since $\varpi$ and $\varpi'$ are coprime, $\varpi'$ does not ramify in $\Anz$. Hence the ramification index of $p$ in $\Hno(\sqrt{\varpi})$ is at most $2$, and this completes the proof.       
\end{proof}

\subsection{Explicit constructions of $H_4$ and $H_8$}
Recall from the discussion at the end of Section 2.1 that $4$ divides $h$ if and only if the prime ideal $\pt$ lying over $2$ splits in $H_2$, which happens if and only if $(1+i)$ splits in $H_2/\QQ(i)$. As $H_2$ is obtained from $\QQ(i)$ by adjoining a square root of $p$, Lemma \ref{ptunram} implies that this happens if and only if $p\equiv \pm 1 \pmod{\mm^5}$, which, for $p\equiv 1 \pmod{4}$, is true if and only if $p\equiv 1\pmod{8}$. Thus we have recovered the criterion for divisibility by $4$. 
\\\\
From now on, assume that $4$ divides $h$, i.e. that $p\equiv 1 \pmod{8}$. We will now use Lemma \ref{lemH2n} to construct the $4$-Hilbert class field of $\QQP$.
\\\\
A prime $p \equiv 1\pmod{4}$ splits in $\QQ(i)$, so that there exists $\pi$ in $\ZZ[i]$ such that $p = \pi\overline{\pi}$; here $\overline{\pi}$ denotes the conjugate of $\pi$ over $A_1 := \QQ$. If we write $\pi$ as $a+bi$
with $a$ and $b$ integers, then we see that $p = a^2 + b^2$. We choose $\pi$ so that $b$ is even. As $p\equiv 1\pmod 8$, we see that $b$ is in fact divisible by $4$. Hence  
\begin{equation}\label{pib}
\pi = a+bi,\ \ \ b\equiv 0\bmod 4.
\end{equation}
Now fix a square root of $\pi$ and denote it by $\sqrt{\pi}$. Recall that $H_2 = \QQ(i, \sqrt{p})$ is the $2$-Hilbert class field of $\QQP$. We claim that the hypotheses of Lemma \ref{lemH2n} for $n = 2$ are satisfied with $A_2:=\QQ(i)$ and $\varpi = \pi$. 
\\\\
All of the hypotheses other than $(U_2)$ and $(N)$ are easy to check. Note that our choice of $\pi$ ensures that $\pi \equiv \pm 1\pmod 4$, so that $(U_2)$ follows from Lemma \ref{ptunram}. To see that $(N)$ is satisfied, note that $H_2(\sqrt{\pi})$ is the splitting field (over $\QQ$) of the polynomial $f_4(X):=(X-\pi)(X-\overline{\pi})$. Indeed, $\pi\overline{\pi}$ is a square in $H_2$, so both square roots of $\overline{\pi}$ are also contained in $H_2(\sqrt{\pi})$. Hence we conclude by Lemma \ref{lemH2n} that the $4$-Hilbert class field is given by  
\begin{equation}
H_4 = H_2(\sqrt{\pi}) = \QQ(i, \sqrt{p}, \sqrt{\pi})
\end{equation}
with $\pi$ as in \eqref{pib}.    
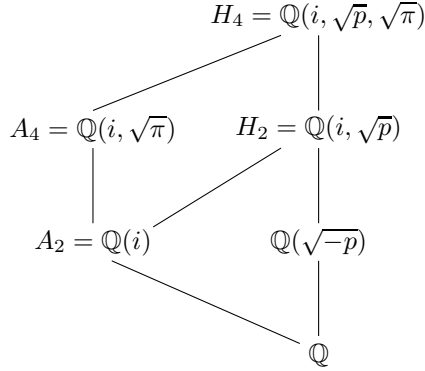
\begin{center} 
\begin{tikzpicture}
  \draw (0, 0) node[]{$H_4 = \QQ(i, \sqrt{p}, \sqrt{\pi})$};
  \draw (0, -1.5) node[]{$H_2 = \QQ(i, \sqrt{p})$};
	\draw (-3, -1.5) node[]{$A_4 = \QQ(i, \sqrt{\pi})$};
  \draw (0, -3) node[]{$\QQ(\sqrt{-p})$};
	\draw (-3, -3) node[]{$A_2 = \QQ(i)$};
	\draw (0, -4.5) node[]{$\QQ$};
  \draw (-0, -4.25) -- (0, -3.25);
	\draw (0, -2.75) -- (0, -1.75);
	\draw (0, -1.25) -- (0, -0.25);
  \draw (-0.25, -4.35) -- (-2.75, -3.25);
	\draw (-3, -2.75) -- (-3, -1.75);
	\draw (-3, -1.25) -- (-0.5, -0.25);
	\draw (-2.2, -2.8) -- (-0.5, -1.75);
\end{tikzpicture}
\end{center}
Next, we find a criterion for divisibility by $8$. Recall that $h$ is divisible by $8$ if and only if $\pt$ splits completely in $H_4$, i.e. if and only if $\pi$ is a square in $\QQ_2(i)$. By Lemma \ref{ptunram}, this happens if and only if $\pi \equiv \pm 1\pmod{\mm^5}$. In terms of $a$ and $b$ from \eqref{pib}, this means that
$$8|h\Longleftrightarrow a+b\equiv \pm 1\bmod{8}.$$
We remark that Fouvry and Kl\a"{u}ners developed similar methods in \cite{FK1}, where they constructed an analogue of the $4$-Hilbert class field to deduce a criterion for the $8$-rank of class groups in a family of real quadratic number fields. From now on, suppose that $8|h$. Replacing $\pi$ by $-\pi$ if necessary, we assume that
\begin{equation}\label{pim}
\pi\equiv 1\pmod{\mm^5}.
\end{equation}
This means that $a+b\equiv 1\pmod{8}$. Our choice of $\sqrt{\pi}$ above is only unique up to sign. By Hensel's lemma, we can now fix this sign by imposing that
\begin{equation}\label{pisq}
\sqrt{\pi}\equiv 1\pmod{\mm^3}.
\end{equation}
In order to explicitly generate $H_8$ from $H_4$ using Lemma \ref{lemH2n}, we are led to the problem of finding a prime element in $A_4 = \QQ(i, \sqrt{\pi})$ whose norm down to $\QQ(i)$ is $\overline{\pi}$, up to units. This is the problem that we cannot solve explicitly enough in general to answer questions about infinitude or density.
\\\\
However, for a very thin subset of primes, we can write down an element of $A_4$ of norm $-\overline{\pi}$. These are primes $p$ of the form 
\begin{equation}\label{pFI}
p = a^2+c^4,\ \ c\text{ even},
\end{equation}
that is, primes $p$ of the form $a^2+b^2$ with $b$ a perfect square divisible by $4$.
\\\\
Suppose that $p$ is a prime of the form \eqref{pFI}. Set 
\begin{equation}\label{varpi}
\varpi_0 = c(1+i)+\sqrt{\pi}.
\end{equation}
Let $\sigma$ be a generator for $\Gal(H_4/\QQP)$ and set $\varpi_2 = \sigma^2(\varpi_0)$. Then
\begin{equation}\label{normvarpi}
\varpi_0\cdot\varpi_2 = (c(1+i)+\sqrt{\pi})(c(1+i)-\sqrt{\pi}) = -\overline{\pi}.
\end{equation}
We can now prove the main result of this section.
\begin{prop}\label{propHeight}
Let $p$ is a prime of the form \eqref{pFI}, let $\pi$ be as in \eqref{pim}, let $\sqrt{\pi}$ be as in \eqref{pisq}, and let $\varpi$ be as in \eqref{varpi}. Let $\sqrt{\varpi}$ denote a square root of $\varpi$. Then $H_4(\sqrt{\varpi})$ is the $8$-Hilbert class field of $\QQP$.
\end{prop}
\begin{proof}
We again use Lemma \ref{lemH2n}, but this time with $n = 3$, $A_4 = \QQ(i, \sqrt{\pi})$ and $\varpi = \varpi_0$. All of the hypotheses except for $(U_2)$ and $(N)$ immediately follow from the identity \eqref{normvarpi}.

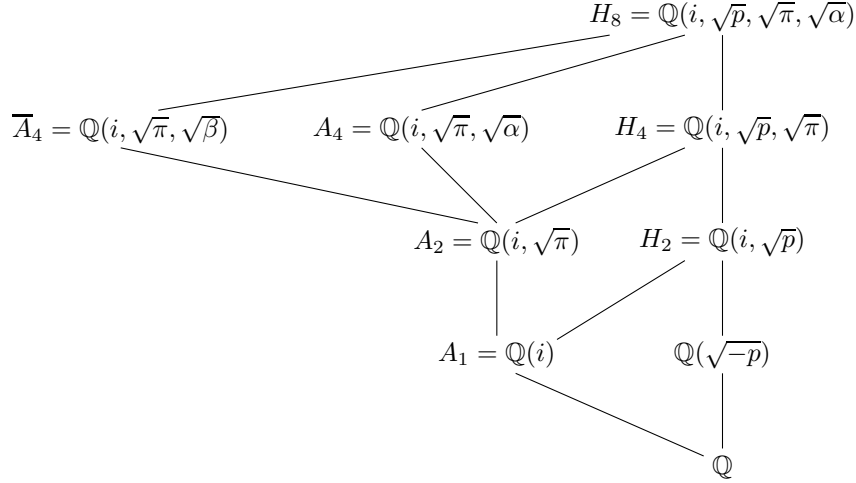
\begin{center} 
\begin{tikzpicture}
  \draw (0, 1.5) node[]{$H_8 = \QQ(i, \sqrt{p}, \sqrt{\pi}, \sqrt{\alpha})$};
	\draw (0, 0.25) -- (0, 1.25);
	\draw (-4, 0) node[]{$A_4 = \QQ(i, \sqrt{\pi}, \sqrt{\alpha})$};
	\draw (-8, 0) node[]{$\overline{A}_4 = \QQ(i, \sqrt{\pi}, \sqrt{\beta})$};
	\draw (0, 0) node[]{$H_4 = \QQ(i, \sqrt{p}, \sqrt{\pi})$};
  \draw (0, -1.5) node[]{$H_2 = \QQ(i, \sqrt{p})$};
	\draw (-3, -1.5) node[]{$A_2 = \QQ(i, \sqrt{\pi})$};
  \draw (0, -3) node[]{$\QQ(\sqrt{-p})$};
	\draw (-3, -3) node[]{$A_1 = \QQ(i)$};
	\draw (0, -4.5) node[]{$\QQ$};
  \draw (0, -4.25) -- (0, -3.25);
	\draw (0, -2.75) -- (0, -1.75);
	\draw (0, -1.25) -- (0, -0.25);
  \draw (-0.25, -4.35) -- (-2.75, -3.25);
	\draw (-3, -2.75) -- (-3, -1.75);
	\draw (-2.75, -1.25) -- (-0.5, -0.25);
	\draw (-3, -1.25) -- (-4, -0.25);
	\draw (-3.25, -1.25) -- (-8, -0.25);
	\draw (-4, 0.25) -- (-0.5, 1.25);
	\draw (-7.5, 0.25) -- (-1.5, 1.25);
	\draw (-2.2, -2.8) -- (-0.5, -1.75);
\end{tikzpicture}
\end{center}

We now prove hypothesis $(N)$. For $0\leq m\leq 3$, set $\varpi_m = \sigma^m(\varpi)$, where $\sigma$ is a generator for $\Gal(H_4/\QQP)$. We claim that $H_4(\sqrt{\varpi_0})$ is the splitting field of the polynomial
$$f_8(X) = (X^2-\varpi_0)(X^2-\varpi_1)(X^2-\varpi_2)(X^2-\varpi_3).$$
It is easy to see that $\varpi_0\varpi_2 = -\overline{\pi}$ and $\varpi_1\varpi_3 = -\pi$ are squares in $H_4$. To prove $(N)$, it now suffices to show that $\varpi_0\varpi_1$ is a square in $H_4$. Write $\sqrt{\pi} = d+ei$ with
$$d = \frac{\sqrt{\pi}+\sqrt{\sigma(\pi)}}{2},\ e=\frac{\sqrt{\pi}-\sqrt{\sigma(\pi)}}{2i}\in H_4.$$ 
One can now check that $\varpi_0\varpi_1 = (c+di-ei)^2$, which proves hypothesis $(N)$.
\\\\
It remains to prove hypothesis $(U_2)$. The assumption that $\pi\equiv 1\pmod{\mm^5}$ actually means that $\pi$ is a square in $\QQ_2(i)$, i.e. that $(1+i)$ splits in $A_4$. Hence it remains to show that $\QQ_2(i, \sqrt{\varpi_0})$ is unramified over $\QQ_2(i)$, and Lemma \ref{ptunram} implies that it is enough to prove that $\varpi_0 \equiv \pm 1\pmod{\mm^4}$.
\\\\
Recall from \eqref{pisq} that $\sqrt{\pi}\equiv 1\pmod{\mm^3}$. Thus $\sqrt{\pi}\equiv 1$ or $1 +m^3 \pmod{\mm^4}$. Squaring, we find that $\pi \equiv 1$ or $1+m^5\pmod{\mm^6}$, respectively.
\\\\
The two cases above correspond to the residue class of $c$ modulo $4$. First recall that $a+b\equiv 1\bmod 8$, i.e. $a+c^2\equiv 1\pmod{\mm^6}$. In the first case, if $c\equiv 0 \pmod{\mm^4}$, then $c^2\in\mm^6$, so $a-1\in\mm^6$ as well. Then $\pi = a+c^2i\equiv 1\pmod{\mm^6}$, which means that $\sqrt{\pi}\equiv 1\pmod{\mm^4}$. Then
$$\varpi_0 = c(1+i)+\sqrt{\pi} \equiv 1 \pmod{\mm^4}.$$
In the second case, $c \equiv 2 \pmod{\mm^4}$, so that $c^2\equiv -m^4\pmod{\mm^6}$. Then $a-1+m^4\in\mm^6$, so that $\pi = a+c^2i\equiv 1-m^4-m^4i \equiv 1+m^4(-1-i)\equiv 1 + m^5\pmod{\mm^6}$. This means that $\sqrt{\pi}\equiv 1 + m^3\pmod{\mm^4}$. Finally,
$$\varpi_0 = \sqrt{\pi} + c(1+i) \equiv 1 + m^3 + m^3 \equiv \pm 1 \pmod{\mm^4},$$
which proves that $\QQ_2(i, \sqrt{\varpi_0})$ is unramified over $\QQ_2(i)$.
\end{proof}

\section{Proof of Theorem 1}
The proof of Theorem 1 will proceed in much the same way as the last part of the proof of Proposition \ref{propHeight}. Now, instead of showing that $\QQ_2(i, \sqrt{\varpi_0})$ is unramified over $\QQ_2(i)$, we must decide when this extension is trivial (i.e. when $\pt$ splits completely in $H_8$) and when it is unramified of degree $2$ (i.e. when $\pt$ does not split completely in $H_8$). This is equivalent to determining when $\varpi_0$ is a square in $\QQ_2(i)$.
\\\\
We will distinguish between two cases as above. The first case is when $c\equiv 0\pmod{4}$, i.e. $c\in\mm^4$. Recall from above that then $a\equiv 1\pmod{8}$ and $\sqrt{\pi}\equiv 1\pmod{\mm^4}$. 
\\\\
To check whether or not $\varpi_0$ is a square in $\QQ_2(i)$, we must compute $\varpi_0$ modulo $\mm^5$. Since $c\equiv 0\pmod{4}$, we deduce that $\varpi_0\equiv \sqrt{\pi}$ modulo $\mm^5$. Thus, we must determine conditions on $a$ such that $\sqrt{\pi}\equiv \pm 1\pmod{\mm^5}$, and for this, by Hensel's lemma, it is necessary to determine $\pi$ modulo $\mm^7$. Hence, assuming $c\equiv 0\pmod{4}$,
\begin{tabbing}
   \hspace{1.5in}\= \hspace{0.5in}\= \hspace{0.5in}\= \kill
   \>$16|h$ \> $\Longleftrightarrow$ \> $\sqrt{\pi}\equiv \pm 1\pmod{\mm^5}$\\ 
	 \>\> $\Longleftrightarrow$\> $\pi\equiv 1\pmod{\mm^7}$\\
	 \>\> $\Longleftrightarrow$\> $a\equiv 1\pmod{16}$.
\end{tabbing}\vspace{0pt}
This proves parts (i) and (iii) of Theorem 1.
\\\\
We handle the second case similarly.  Now $c\equiv 2\pmod{4}$, $a\equiv 5\pmod{8}$ and $\sqrt{\pi}\equiv 1+m^3\pmod{\mm^4}$. Then $\alpha\equiv 2m + \sqrt{\pi}$ modulo $\mm^5$ and so we must determine conditions on $a$ such that $\sqrt{\pi}\equiv \pm 1 -2m\pmod{\mm^5}$. Under the current assumptions,
\begin{tabbing}
   \hspace{1.5in}\= \hspace{0.5in}\= \hspace{0.5in}\= \kill 
   \>$16|h$ \> $\Longleftrightarrow$ \> $\sqrt{\pi}\equiv \pm 1-2m\pmod{\mm^5}$\\ 
	 \>\> $\Longleftrightarrow$\> $\pi\equiv 1+m^5+m^6\pmod{\mm^7}$\\
	 \>\> $\Longleftrightarrow$\> $a\equiv -3\pmod{16}$.
\end{tabbing}\vspace{0pt}
Note that because of the choice \eqref{pim} we have actually shown the theorem for $a\equiv 1\pmod{4}$. If $p = a^2+c^4$ with $a\equiv 3\pmod{4}$, then $p = (-a)^2+c^4$ with $-a\equiv 1\pmod{4}$, so that the other cases can be deduced immediately. This finishes the proof of Theorem 1. 

\section{Overview of the proof of Proposition 1}
In \cite{FI1}, Friedlander and Iwaniec prove an asymptotic formula for the number of primes of the form $a^2+c^4$, that is, primes of the form $a^2+b^2$ where $b$ itself is a square. For a summary of their proof, see the exposition in \cite[Chapter 21]{FI3}. They use a new sieve that they developed to detect primes in relatively thin sequences \cite{FI2}. This sieve has its roots in the work of Fouvry and Iwaniec \cite{FoIw}, where they used similar sieve hypotheses to give an asymptotic formula for the number of primes of the form $a^2+b^2$ where $b$ is a prime.
\\\\
The purpose of the following three sections is to demonstrate that the method of Friedlander and Iwaniec is robust enough to incorporate congruence conditions on $a$ and $c$. While we are convinced that Proposition 1 remains true when $a$ and $c$ satisfy reasonable congruence conditions modulo any positive integers $q_1$ and $q_2$, respectively, the technical obstacles necessary to insert the congruence condition for $c$ are cumbersome. Hence we will restrict ourselves to the case $q_2 = 4$.  
\\\\
The proof of Proposition 1 involves certain alterations in the way that the sieve \cite{FI2} is used. For this reason, we first briefly recall the inputs and the output of the sieve.

\subsection{Asymptotic sieve for primes}
Suppose $(a_n)$ ($n\in \NN$) is a sequence of non-negative real numbers. Then the asymptotic sieve for primes developed in \cite{FI2} yields an asymptotic formula for 
$$S(x) = \sum_{\substack{p\leq x \\ p\text{ prime}}}a_p\log p$$
provided that the sequence $(a_n)$ satisfies several hypotheses, all but two of which are not difficult to verify. To state them, we first need to fix some terminology. For $d\geq 1$, let 
$$A_d(x) = \sum_{\substack{ n\leq x\\ n\equiv 0\bmod d }}a_n$$
and let $A(x) = A_1(x)$. Moreover, let $g$ be a multiplicative function, and define the \textit{error term} $r_d(x)$ by the equality
\begin{equation}\label{rdx}
A_d(x) := g(d)A(x) + r_d(x).
\end{equation}
The hypotheses which are not difficult to verify are listed in equations (2.1)-(2.8) in \cite{FI1}. We briefly recall them here. We assume the bounds
\begin{equation}\tag{H1}
A(x)\gg A(\sqrt{x})(\log x)^2
\end{equation}
and
\begin{equation}\tag{H2}
A(x)\gg x^{\frac{1}{3}}\left(\sum_{n\leq x}a_n^2\right)^{\frac{1}{2}}.
\end{equation}
We assume that the multiplicative function $g$ satisfies 
\begin{equation}\tag{H3}
0\leq g(p^2)\leq g(p)\leq 1,
\end{equation}
\begin{equation}\tag{H4}
g(p)\ll p^{-1},
\end{equation}
and
\begin{equation}\tag{H5}
g(p^2)\ll p^{-2}.
\end{equation}
We also assume that for all $y\geq 2$, 
\begin{equation}\tag{H6}
\sum_{p\leq y}g(p) = \log\log y + c + O((\log y)^{-10}),
\end{equation}
where $c$ is a constant depending only on $g$; this is the linear sieve assumption. Finally, we assume the bound
\begin{equation}\tag{H7}
A_d(x)\ll d^{-1}\tau(d)^8A(x)
\end{equation}
uniformly in $d\leq x^{\frac{1}{3}}$; here $\tau$ is the divisor function.  
\\\\
Now we state the two hypotheses which are more difficult to verify. The first is a classical sieve hypothesis; it is a condition on the average value of the error terms $r_d(x)$. Let $L = (\log x)^{2^{24}}$. 
\begin{hypoR}
There exists $x_r>0$ and $D = D(x)$ in the range
\begin{equation}\label{Drange}
x^{\frac{2}{3}}<D<x
\end{equation}
such that for all $x\geq x_r$, we have 
\begin{equation}\tag{R}
\sum_{\substack{d\text{ cubefree} \\ d\leq DL^2}}|r_d(t)|\leq A(x)L^{-2}
\end{equation}
uniformly in $t\leq x$.
\end{hypoR}
In our applications, $D$ will be $x^{3/4-\varepsilon}$ for a sufficiently small $\varepsilon$. This condition about \textit{remainders} will be called condition (R). 
\\\\
The second is a complicated condition on bilinear forms in the elements of the sequence $(a_n)$ weighed by truncated sums of the M\a"obius function
\begin{equation}\label{betanc}
\beta(n, C) = \mu(n)\sum_{c|n,\ c\leq C}\mu(c).
\end{equation}
It is designed to make sure that the sequence $(a_n)$ is orthogonal to the M\a"obius function; this is crucial in overcoming the parity problem. We now state this hypothesis, named (B) for \textit{bilinear}.
\begin{hypoB}
Suppose (R) is satisfied for $x_r$ and $D = D(x)$. Then there exists $x_b>x_r$ such that for every $x>x_b$, there exist $\delta$, $\Delta$, and $P$ satisfying
$$2\leq \delta\leq \Delta,$$
$$2\leq P\leq \Delta^{1/2^{35}\log\log x},$$
and such that for every $C$ with 
$$1\leq C\leq xD^{-1},$$
and for every $N$ with 
$$\Delta^{-1}\sqrt{D}<N<\delta^{-1}\sqrt{x},$$
we have
\begin{equation}\tag{B}
\sum_{m}\left|\sum_{\substack{N\leq n\leq 2N \\ mn\leq x \\ (n, m\Pi) = 1}}\beta(n, C)a_{mn}\right|\leq A(x)(\log x)^{-2^{26}},
\end{equation}
where 
\begin{equation}\label{pi}
\Pi = \prod_{p\leq P}p.
\end{equation}
\end{hypoB}
Note that establishing condition (R) for a larger $D$ decreases the range of $C$ and $N$ for which we have to verify condition (B).  
\\\\
The main result of \cite{FI2} is
\begin{theorem}\label{ASP}
Assuming hypotheses (H1)-(H7), (R), and (B), we have
$$S(x) = HA(x)\left(1+O\left(\frac{\log \delta}{\log \Delta}\right)\right),$$
where $H$ is the positive constant given by the convergent product 
$$H = \prod_p(1-g(p))\left(1-\frac{1}{p}\right)^{-1}$$
and the constant implied in the O-symbol depends on the function $g$ and the constants implicit in (H1), (H2), and (H7). 
\end{theorem}

\subsection{Preparing the sieve for Proposition 1}
For our application, we will denote by $v'$ the analogue of a quantity $v$ from the proof of Friedlander and Iwaniec in \cite{FI1}. We take $(a'_n)$ to be the following sequence. Suppose $q_1$ and $q_2$ are positive integers and let $q$ denote the least common multiple of $q_1$ and $q_2$. We say that a pair of congruence classes 
$$a_0\bmod q_1\ \ \ \ \ \ \ \ c_0 \bmod q_2$$ 
is \textit{admissible} if for every pair congruence classes
$$a_1\bmod q\ \ \ \ \ \ \ \ c_1 \bmod q$$
such that $a_1\equiv a_0\bmod q_1$ and $c_1\equiv c_0\bmod q_2$, the congruence class $a_1^2+c_1^4 \bmod q$ is a unit modulo $q$.
\begin{example}
Suppose that $a_0\in\{1, 3, 5, 7, 9, 11, 13, 15\}$ and $c_0\in\{0, 2\}$. Then the pair of congruence classes $a_0\bmod 16$ and $c_0\bmod 4$ is admissible.
\end{example}
\begin{example}
Suppose that $a_0=c_0=1$. Then the pair of congruence classes $a_0\bmod 3$ and $c_0\bmod 2$ is \textit{not} admissible. Indeed, $1\equiv a_0 \equiv c_0 \bmod 6$ but $2\equiv 1^2+1^4 \bmod 6$ is not invertible modulo $6$. This does not mean, however, that there are no primes of the form $a^2+c^4$ with $a\equiv 1\bmod 3$ and $c\equiv 1\bmod 2$; one such prime is $4^2+1^4$.
\end{example}
Henceforth, suppose $q_1$ and $q_2$ are positive integers, let $q$ be the least common multiple of $q_1$ and $q_2$, and suppose $a_0\bmod q_1$ and $c_0\bmod q_2$ is an admissible pair of congruence classes. We define
\begin{equation}\label{anprime}
a'_n := \sumsum_{\substack{a,\ b\ \in\ \ZZ \\ a^2+b^2 = n \\ a\equiv a_0\bmod q_1}}\zz'(b),
\end{equation}
where
\begin{equation}\label{zzprime}
\zz'(b) := \sum_{\substack{c\in\ZZ \\ c^2 = b \\ c\equiv c_0\bmod q_2}}1.
\end{equation}
Let $g$ be the multiplicative function supported on cubefree integers defined in \cite[Equation 3.16, p.961]{FI1} as follows: let $\chi_4$ denote the character of conductor $4$; for $p\geq 3$ set
$$g(p)p = 1+\chi_4(p)\left(1-\frac{1}{p}\right)$$
and
$$g(p^2)p^2 = 1+(1+\chi_4(p))\left(1-\frac{1}{p}\right);$$
finally, set $g(2) = \frac{1}{2}$ and $g(4) = \frac{1}{4}$.
For our extension, we define a multiplicative function $g'$ by setting 
$$
g'(n) =
\begin{cases}
g(n) & \text{if }(n, q) = 1 \\
0 & \text{otherwise.}
\end{cases}
$$
Then, provided that (H1)-(H7), (R), and (B) are satisfied with $\delta$ a large power of $\log x$ and $\Delta$ a small power of $x$, the asymptotic formula given by the sieve (see Theorem \ref{ASP}) is 
\begin{equation}\label{asymp}
S'(x) := \sum_{\substack{p\leq x\\p\text{ prime}}}a'_p\log p = c(q_1, q_2)\frac{16\kappa}{\pi} x^{3/4}\left(1+O\left(\frac{\log\log x}{\log x}\right)\right)
\end{equation}
where 
$$c(q_1, q_2) = \frac{1}{q_1q_2}\prod_{p|q}(1-g(p))^{-1}$$
and $\kappa$ is the integral given in the statement of Proposition 1. Note that the sieve applied to the original sequence $(a_n)$ from \cite{FI1}, with
\begin{equation}\label{an}
a_n = \sumsum_{\substack{a,\ b\ \in\ \ZZ \\ a^2+b^2 = n}}\zz(b),
\end{equation}
where
\begin{equation}\label{zz}
\zz(b) = \sum_{\substack{c\in\ZZ \\ c^2 = b}}1,
\end{equation}
yields the asymptotic formula
$$S(x) = \frac{16\kappa}{\pi} x^{3/4}\left(1+O\left(\frac{\log\log x}{\log x}\right)\right)$$
(see \cite[Theorem 1, p.946]{FI1}). Thus $c(q_1, q_2)$ can be interpreted as the density of primes of the form $a^2+c^4$ such that $a\equiv a_0\bmod q_1$ and $c\equiv c_0\bmod q_2$ within the set of all primes of the form $a^2+c^4$. 
\begin{remark}
Throughout the following two sections, we regard $q_1$ and $q_2$ as fixed constants, and so the implied constants in every bound we give may depend on $q_1$ and $q_2$, even if this dependence is not explicitly stated. Thus, whenever we state ``the implied constant is absolute,'' the implied constant may actually depend on $q_1$ and $q_2$. In our application $q_1 = 16$ and $q_2 = 4$, so we are not concerned with uniformity of the above asymptotic formula with respect to $q_1$ and $q_2$. 
\end{remark}
It is obvious that our modified sequence $(a'_n)$ satisfies (H1)-(H7) for the same reasons as the original sequence $(a_n)$. We will prove that $(a'_n)$ above satisfies condition (R) for general $q_1$ and $q_2$. The congruence condition on $c$ is more difficult to insert into the proof of condition (B), so we prove condition (B) only for the special case where $q_2 = 4$ and $c_0\in\{0, 2\}$.  
   
\section{Proof of condition (R)}
Here we closely follow and refer to the arguments laid out in \cite[Section 3, p.955-962]{FI1}. Define
$$A'_d(x) := \sum_{\substack{n\leq x\\ n\equiv 0\bmod d}}a'_n$$
and
$$A'(x):= A'_1(x).$$
The goal is to check that the error terms $r'_d(x)$ defined by 
\begin{equation}\label{defrdx}
r'_d(x) :=  A'_d(x) - g'(d)A'(x)
\end{equation}
are small on average. To do this, we will prove an analogue of \cite[Lemma 3.1, p.956]{FI1}, with $M_d(x)$ (representing the \textit{main term} and defined in \cite[p.955]{FI1}) replaced by
$$M'_d(x) = \frac{1}{dq_1}\sumsum_{0<a^2+b^2\leq x}\zz'(b)\rho(b; d)\ \ \ \ \ \ \text{ if }(d, q) = 1$$
and $M'_d(x) = 0$ otherwise; here $\rho(b; d)$ is defined as in \cite[p.955]{FI1}, i.e. it is the number of solutions $\alpha\bmod d$ to
$$\alpha^2+b^2\equiv 0\bmod d.$$
We separate the case when $d$ is not coprime to $q$ because in this case $A'_d(x) = 0$. This follows because the pair of congruences $a_0\bmod q_1$ and $c_0\bmod q_2$ is admissible and hence $a'_n$ is supported on $n$ coprime to $q$. The lemma we wish to prove is now identical to \cite[Lemma 3.1, p.956]{FI1}.
\begin{lemma}\label{lem31}
For any $D\geq 1$, any $\varepsilon>0$, and any $x\geq 2$, we have
$$\sum_{d\leq D}|A'_d(x)-M'_d(x)|\ll D^{\frac{1}{4}}x^{\frac{9}{16}+\varepsilon},$$
where the implied constant depends only on $\varepsilon$.
\end{lemma} 
This result is useful because it is easy to obtain an asymptotic formula for $M'_d(x)$ where the coefficient of the leading term is, up to a constant, a nice multiplicative function of $d$. In fact, let $h$ be the multiplicative function supported on cubefree integers defined in \cite[(3.16), p.961]{FI1} by
\begin{equation}\label{defh}
\begin{cases}
h(p)p = 1+2(1+\chi_4(p))\\
h(p^2)p^2 = p+2(1+\chi_4(p)),
\end{cases}
\end{equation}
and define a multiplicative function $h'$ by setting 
\begin{equation}\label{defhprime}
h'(n) =
\begin{cases}
h(n) & \text{if }(n, q) = 1 \\
0 & \text{otherwise.}
\end{cases}
\end{equation}
Then following the same argument as in the proof of \cite[Lemma 3.4, p.961]{FI1}, we get
\begin{lemma}\label{lem34}
For $d$ cubefree we have
$$M'_d(x) = g'(d)\frac{4\kappa x^{\frac{3}{4}}}{q_1q_2}+O\left(h'(d)x^{\frac{1}{2}}\right),$$
where $\kappa$ is the integral given in the statement of Proposition 1 and the implied constant is absolute. $\Box$
\end{lemma}
Combining Lemmas $\ref{lem31}$ and $\ref{lem34}$, we get, as in \cite[Proposition 3.5, p.362]{FI1},
\begin{prop}\label{prop35}
Let 
$$a_0\bmod q_1\ \ \ \ \ \ \ \ c_0 \bmod q_2$$ 
be an admissible pair of congruence classes, let $a'_n$ be defined as in \eqref{anprime}, and let $r'_d(x)$ be defined as in \eqref{defrdx}. Then for every $\varepsilon >0$ and every $D\geq 1$, there exists an $x_0 = x_0(\varepsilon)>0$ and $C = C(\varepsilon)>0$ such that for every $x\geq x_0$, we have 
$$\sum_{\substack{d\text{ cubefree} \\ d\leq D}}|r'_d(t)|\leq C D^{\frac{1}{4}}x^{\frac{9}{16}+\varepsilon}$$
uniformly for $t\leq x$.
\end{prop}
Choosing $D = x^{\frac{3}{4}-8\varepsilon}$, we obtain hypothesis (R). 
\\\\
It remains to prove Lemma $\ref{lem31}$. We may assume that the sum is over $d\leq D$ with $(d, q) = 1$. For such $d$, we first approximate the sum $A'_d(x)$ by a smoothed sum
$$A'_d(f) = \sum_{n\equiv 0\bmod d}a'_nf(n),$$
where $f$ is a smooth function satisfying: 
\begin{itemize}
	\item $f$ is supported on $[0, x]$,
	\item $f(u) = 1$ for $0< u \leq x-y$, 
	\item $f^{(j)}(u)\ll y^{-j}$ for $x-y<u<x$, 
\end{itemize}
where $y = D^{\frac{1}{4}}x^{\frac{13}{16}}$ and the implied constants depend only on $j$ (see \cite[p.958]{FI1}). Since $a'_n$ is supported on integers of the form $a^2+c^4$, we trivially have
$$\sum_{\substack{d\leq D\\ (d, q) = 1}}|A'_d(x)-A'_d(f)|\ll yx^{-\frac{1}{4}+\varepsilon},$$
where the implied constant depends only on $\varepsilon$. With the above choice of $y$, it remains to prove Lemma \ref{lem31} with $A'_d(x)$ replaced by $A'_d(f)$. Similarly as on \cite[p.958]{FI1}, we write
\begin{equation}\label{Adf}
A'_d(f) = \sum_{b}\zz'(b)\sum_{\substack{\alpha\bmod d \\ \alpha^2+b^2\equiv 0\bmod d}}\sum_{\substack{a\equiv \alpha\bmod d \\ a\equiv a_0\bmod q_1}}f(a^2+b^2).
\end{equation}
Since $(d, q) = 1$, so also $(d, q_1) = 1$, and the two conditions $a\equiv \alpha\bmod d$ and $a\equiv a_0\bmod q_1$ can be combined into one condition $a\equiv \alpha' \bmod dq_1$. In fact, fixing an integer $\overline{d}$ that is an inverse of $d$ modulo $q_1$ and an integer $\bar{q}_1$ that is an inverse of $q_1$ modulo $d$, we can define $\alpha'$ as
$$\alpha'=\alpha q_1\bar{q}_1 + a_0 d \bar{d}.$$
We apply Poisson's summation formula to the sum over $a$ to obtain
$$\sum_{a\equiv \alpha'\bmod dq_1}f(a^2+b^2) = \frac{1}{dq_1}\sum_{k}e\left(\frac{\alpha'k}{dq_1}\right)\int_{-\infty}^{\infty}f(t^2+b^2)e\left(\frac{-tk}{dq_1}\right)dt.$$
Here and henceforth, we use the standard notation
$$e(t):= e^{2\pi i t}.$$
Substituting this into \eqref{Adf} we get
$$A'_d(f) = \frac{2}{dq_1}\sum_{b}\zz'(b)\sum_{k}\rho'(k, b; d)I(k, b; dq_1)dt,$$
where 
$$\rho'(k, b; d) = \sum_{\substack{\alpha\bmod d \\ \alpha^2+b^2\equiv 0\bmod d}}e\left(\frac{\alpha'k}{dq_1}\right),$$
and where
$$I(k, b; dq_1) = \int_0^{\infty}f(t^2+b^2)\cos(2\pi t k/dq_1)dt$$
is defined exactly the same as on \cite[p.959]{FI1}. We define $M'_d(f)$ to be the main term in this expansion, i.e. the term corresponding to $k = 0$,
$$M'_d(f) = \frac{2}{dq_1}\sum_{b}\zz'(b)\rho(b; d)I(0, b; dq_1).$$
Since $I(0, b; dq_1) = I(0, b, q_1)$, the argument on page 959 shows that
$$\sum_{\substack{d\leq D \\ (d, q) = 1}}|M'_d(f) - M'_d(x)| \ll yx^{-\frac{1}{4}}(\log x)^2\ll D^{\frac{1}{4}}x^{\frac{9}{16}+\varepsilon},$$
where the implied constants depend only on $\varepsilon$. It remains to prove Lemma \ref{lem31} with $A'_d(f)$ in place of $A'_d(x)$ and $M'_d(f)$ in place of $M'_d(x)$, i.e. to show that $M'_d(f)$ is indeed (on average) the main term in the above Fourier expansion of $A'_d(f)$.
\\\\
Following the argument on \cite[p.959-960]{FI1}, we see that it suffices to show an analogue of \cite[Lemma 3.3, p.957]{FI1} for $\rho'(k, l; d)$. 
\begin{lemma}\label{lem33}
For any $D$, $K$, and $L\geq 1$, for any complex numbers $\xi(k, l)$, and for any $\varepsilon>0$, we have the inequality
$$\sum_{d\leq D}\left|\sumsum_{\substack{0<k\leq K \\ 0<l\leq L}}\xi(k, l)\rho'(k, l; d)\right|\ll (D+\sqrt{DKL})(DKL)^{\varepsilon}\|\xi\|$$
where 
$$\|\xi\|^2 = \sumsum_{\substack{0<k\leq K \\ 0<l\leq L}}|\xi(k, l)|^2,$$
and the implied constant depends only on $\varepsilon$.
\end{lemma}
Recall the following inequality from \cite[(3.6), p.957]{FI1}: for any complex numbers $\alpha_n$ and any $D, N\geq 1$, we have
\begin{equation}\label{lsi0}
\sum_{d\leq D}\sum_{\substack{\nu \bmod d\\ \nu^2+1\equiv 0 \bmod d}}\left|\sum_{n\leq N}\alpha_n e\left(\frac{\nu n}{d}\right)\right|\ll D^{\frac{1}{2}}(D+N)^{\frac{1}{2}}\|\alpha\|,
\end{equation}
where
$$\|\alpha\|:=\left(\sum_{n}|\alpha_n|^2\right)^{\frac{1}{2}},$$
and the implied constant is absolute. Lemma \ref{lem33} can be proved in the same way as \cite[Lemma 3.3, p.957]{FI1} given the following analogue of inequality \eqref{lsi0}. 
\begin{lemma}\label{lsi}
Let $D, N\geq 1$ and let $\alpha_n$ be any complex numbers. For integers $d$ such that $(d, q_1) = 1$, let $\nu'$ be an integer in the unique residue class modulo $dq_1$ that reduces to $\nu$ modulo $d$ and $a_0$ modulo $q_1$. Then there exists an absolute constant $C = C(q_1)$ such that for all $D$ and $N$ sufficiently large, we have 
\begin{equation}\label{eqlsi}
\sum_{\substack{d\leq D\\ (d, q_1) = 1}}\sum_{\substack{\nu \bmod d\\ \nu^2+1\equiv 0 \bmod d}}\left|\sum_{n\leq N}\alpha_n e\left(\frac{\nu' n}{dq_1}\right)\right|\leq C D^{\frac{1}{2}}(D+N)^{\frac{1}{2}}\|\alpha\|.
\end{equation} 
\end{lemma}
Inequality \eqref{lsi0} is a consequence of a large sieve inequality applied to the rationals $\nu/d \bmod 1$ with $\nu$ ranging over the roots of $\nu^2+1\equiv 0\bmod d$ for $d$ in a range around $D$. The large sieve inequality can be applied because these rationals $\nu/d$ are well-spaced modulo 1 for $d$ in a certain range around $D$ (i.e. pairwise differences are uniformly bounded from below by about $1/D$ instead of $1/D^2$). This is a key ingredient in the work of \cite{FoIw}. In our analogue, however, it is not clear that $\nu'/dq_1$ are also well-spaced modulo 1 for $d$ in a similar range around $D$. Nonetheless, we can reduce Lemma \ref{lsi} to inequality \eqref{lsi0} as follows.
\\\\
We first split the sum over $n$ into congruence classes modulo $q_1$ to get
$$\sum_{n_0 \bmod q_1}\sum_{\substack{n\leq N\\ n\equiv n_0 \bmod q_1}}\alpha_n e\left(\frac{\nu' n}{dq_1}\right) = \sum_{n_0 \bmod q_1}\sum_{m\leq (N-n_0)/q_1}\alpha_{m, n_0} e\left(\frac{\nu' m}{d}\right)e\left(\frac{\nu' n_0}{dq_1}\right),$$
where
$$\alpha_{m, n_0} = \alpha_{mq_1+n_0}.$$
Since $e\left(\nu' n_0 / dq_1\right)$ does not depends on $m$, the sum on the left-hand-side of \eqref{eqlsi} is
$$\leq \sum_{n_0 \bmod q_1}\sum_{\substack{d\leq D\\ (d, q_1) = 1}}\sum_{\substack{\nu \bmod d\\ \nu^2+1\equiv 0 \bmod d}}\left|\sum_{m\leq (N-n_0)/q_1}\alpha_{m, n_0} e\left(\frac{\nu' m}{d}\right)\right|.$$
Now $e\left(\frac{\nu' m}{d}\right) = e\left(\frac{\nu m}{d}\right)$ and 
$$\sum_{m}|\alpha_{m, n_0}|^2 \leq \sum_{n}|\alpha_n|^2,$$
so that by \eqref{lsi0} we get
$$\sum_{\substack{d\leq D\\ (d, q_1) = 1}}\sum_{\substack{\nu \bmod d\\ \nu^2+1\equiv 0 \bmod d}}\left|\sum_{n\leq N}\alpha_n e\left(\frac{\nu' n}{dq_1}\right)\right|\ll q_1 D^{1/2}(D+N/q_1)^{1/2}\|\alpha\|.$$
This finishes the proof of \eqref{lsi} and thus also the proof of condition (R).

\section{Proof of condition (B)}     
Many of the upper bound estimates carried out in sections 4 and 5 of \cite{FI1} require no changes since $0\leq a'_n\leq a_n$ (compare \eqref{anprime} and \eqref{an}). In most cases, we now sum over fewer non-negative terms.  
\\\\
Recall that we established condition (R) with $D = x^{\frac{3}{4}-8\varepsilon}$. All of the refinements from \cite[Section 4, p.962-966]{FI1} remain valid for our modified sequence $(a'_n)$. We briefly recall these refinements. First note that it is enough to prove the analogue of \cite[Proposition 4.1, p.963]{FI1}:
\begin{prop}\label{prop41}
Let $c_0\in\{0, 2\}$, let $q_2 = 4$, and let
$$a_0\bmod q_1\ \ \ \ \ \ \ \ c_0 \bmod q_2$$ 
be an admissible pair of congruence classes. Define $\beta(n, C)$ as in \eqref{betanc}, $\Pi$ as in \eqref{pi}, and $a'_n$ as in \eqref{anprime}. Let $x\geq 3$, $\eta>0$, and $A>0$. Let $P$ be in the range
\begin{equation}\label{Prange}
(\log\log x)^2\leq \log P\leq (\log x)(\log\log x)^{-2}.
\end{equation}
Let 
\begin{equation}\label{bigB}
B = 4A+2^{20}.
\end{equation}
Then there exists $x_0 = x_0(\eta, A)$ such that for all $x\geq x_0$, for all $N$ with 
\begin{equation}\label{condN}
x^{\frac{1}{4}+\eta}<N<x^{\frac{1}{2}}(\log x)^{-B},
\end{equation}
and for all $C$ with
\begin{equation}\label{eqC}
1\leq C\leq N^{1-\eta},
\end{equation}
we have
\begin{equation}\label{eqprop41}
\sum_{m}\left|\sum_{\substack{N\leq n\leq 2N \\ mn\leq x \\ (n, m\Pi) = 1}}\beta(n, C)a_{mn}\right|\leq A'(x)(\log x)^{5-A}.
\end{equation}
\end{prop}

\subsection{From Propositions \ref{prop35} and \ref{prop41} to Proposition 1}
Before proving Proposition \ref{prop41}, we deduce Proposition 1 from Propositions \ref{prop35} and \ref{prop41}. Let $a_0\in\{1, 3, 5, 7, 9, 11, 13, 15\}$, $q_1 = 16$, $c_0\in\{0, 2\}$, and $q_2 = 4$. Then 
$$a_0\bmod q_1\ \ \ \ \ \ \ \ c_0 \bmod q_2$$
is an admissible pair of congruences. We apply the asymptotic sieve for primes described in Section 4.1 to the sequence $(a'_n)$ defined in \eqref{anprime}. Hypotheses (H1)-(H7) for $(a'_n)$ are verified in the same way as hypotheses (H1)-(H7) for the sequence $(a_n)$ defined in \eqref{an} (see comment at the end of Section 4.2).  
\\\\
Proposition \ref{prop35} implies that $(a'_n)$ satisfies hypothesis (R) for $\varepsilon = 1/8000$, 
\begin{equation}\label{bigD}
D = x^{\frac{3}{4}-\frac{1}{1000}},
\end{equation}
which is indeed in the range \eqref{Drange}, and $x_r = x_r(\varepsilon)$ large enough.
\\\\
Applying Proposition \ref{prop41} with the same $D$ as in \eqref{bigD}, with $P$ any number in the range \eqref{Drange}, with $A = 5+2^{26}$, and with $\eta =\frac{1}{100}$ establishes hypothesis (B) for the sequence $(a'_n)$ with $\delta = (\log x)^B$, $\Delta = x^{\eta}$, and $x_b = \max \{x_r, x_0(\eta, A)\}$.
\\\\
We then obtain the asymptotic formula \eqref{asymp} with
$$c(q_1, q_2) = \frac{1}{32},$$
which proves \eqref{prop1eq}.

\subsection{Proof of Proposition \ref{prop41}}
Suppose that we are in the setting of Proposition \ref{prop41}. Now take $A'= 2A+2^{20}$ (see \cite[p.1018]{FI1}) and define
$$\vartheta := (\log x)^{-A}$$
and
\begin{equation}\label{theta}
\theta := (\log x)^{-A'}
\end{equation}
as on \cite[p.965]{FI1}. We split the sum \eqref{eqprop41} by using a smooth partition of unity. Let $p$ be a smooth function supported on an interval 
$$N'<n\leq (1+\theta)N'$$
with $N<N'<2N$, and suppose that $p$ is twice differentiable with 
$$p^{(j)}\ll(\theta N)^{-j}$$
for $j = 0, 1, 2$ (see \cite[(4.14), p.965]{FI1}). It then suffices to show Proposition \ref{prop41} with $\beta(n, C)$ replaced by a smoothed version
\begin{equation}\label{betancsmooth}
\beta(n) = \beta(n, C) = p(n)\mu(n)\sum_{c|n,\ c\leq C}\mu(c)
\end{equation}
and the bound $\leq A'(x)(\log x)^{5-A}$ replaced by $\leq C\vartheta\theta A'(x)(\log x)^{5}$ (see \cite[(4.17), p.965]{FI1}). Moreover, one can split the sum over $m$ in \eqref{eqprop41} into dyadic segments $M\leq m\leq 2M$ with $M$ satisfying
\begin{equation}\label{condMN}
\vartheta x\leq MN\leq x.
\end{equation}
We remark that \eqref{condN} now implies that $N\leq \vartheta\theta(MN)^{\frac{1}{2}}$. Sums over the remaining dyadic segments are bounded trivially at an acceptable cost. Again, for an acceptable cost, one can suppose that $\beta(n, C)$ is supported on $n$ with
\begin{equation}\label{taun}
\tau(n)\leq \tau := (\log x)^{A+2^{20}}.
\end{equation}
(see \cite[p.963-966, 1018]{FI1}). For convenience of notation, we also restrict the support of $\beta(n, C)$ to $n$ satisfying
\begin{equation}\label{coprimePi}
(n, \Pi) = 1,
\end{equation}
where $\Pi$ is defined in \eqref{pi}. Finally, let $\alpha(m)$ be any complex numbers supported on $M<m\leq 2M$ with $|\alpha(m)|\leq 1$, and define
\begin{equation}\label{Bprime}
\mathcal{B}'^{\ast}(M, N) := \sumsum_{(m, n) = 1}\alpha(m)\beta(n)a'_{mn},
\end{equation}
where $\beta(n) = \beta(n, C)$ is defined as in \eqref{betancsmooth} (see \cite[(4.20), p.966]{FI1}). To establish condition (B) it then suffices to prove
\begin{lemma}\label{reduct1}
Let $\eta>0$ and $A>0$ and take $B$ as in \eqref{bigB}. Then there exists $x_0 = x_0(\eta, A)>0$ such that for all $x\geq x_0$, for all $M$ and $N$ satisfying \eqref{condN} and \eqref{condMN}, and for all $C$ satisfying \eqref{eqC} we have
\begin{equation}\label{eqreduct1}\tag{B'}
\left|\mathcal{B}'^{\ast}(M, N)\right|\leq \vartheta\theta(MN)^{\frac{3}{4}}(\log MN)^5.
\end{equation}
\end{lemma}

\subsection{Proof of Lemma \ref{reduct1}}
In \cite[Section 5]{FI1}, one begins to exploit the arithmetic in $\ZZ[i]$ and the inequality (B') is reduced to another inequality involving sums over Gaussian integers. In our context, where $a'_n$ are defined in \eqref{anprime}, equation \cite[(5.2), p.967]{FI1} now becomes (for $(m, n) = 1$)
$$a'_{mn} = \underset{\Imag \overline{w}z \equiv a_0\bmod q_1}{\sum_{|w|^2 = m}\sum_{|z|^2 = n}}\zz'(\Real \overline{w}z),$$
where the sum over $z$ is restricted to primary Gaussian integers, i.e. $z$ satisfying
$$z\equiv 1\bmod 2(1+i).$$
Recall from \eqref{zzprime} that the congruence condition $c\equiv c_0\bmod q_2$ is incorporated into the definition of $\zz'$. We now define $\alpha_w := \alpha(|w|^2)$ and $\beta_z := \beta(|z|^2)$ as on \cite[p.967]{FI1}, so that \eqref{Bprime} becomes
\begin{equation}\label{Bzw}
\mathcal{B}'^{\ast}(M, N) = \sumsum_{\substack{(w\overline{w}, z\overline{z}) = 1\\ \Imag \overline{w}z \equiv a_0\bmod q_1}}\alpha_w\beta_z \zz'(\Real \overline{w}z).
\end{equation}
Similarly as in \cite[(5.7), p.967]{FI1}, we split the sum $\mathcal{B}'^{\ast}(M, N)$ into $O(q_1^4)$ sums by restricting the support of $\alpha_w$ to $w$ in a fixed residue class modulo $q_1$ and $\beta_z$ to $z$ in a fixed residue class $z_0$ modulo $64q_1$, such that $z_0 \equiv 1\bmod{2(1+i)}$. Now the residue class of $\Imag \overline{w}z$ modulo $q_1$ is fixed, and so we can eliminate the condition $\Imag \overline{w}z \equiv a_0\bmod q_1$. 
\\\\
We further modify the support of $\beta_z$ as in equation \cite[(5.13), p.969]{FI1}. Let $r(\alpha)$ be a smooth periodic function of period $2\pi$ supported on $\varphi<\alpha\leq \varphi + 2\pi\theta$ (where $\theta$ is as defined in \eqref{theta})  for some $-\pi<\varphi<\pi$ such that $r^{(j)}\ll\theta^{-j}$ for $j = 0, 1, 2$, and let
\begin{equation}\label{betaz}
\beta_z = r(\alpha)p(n)\mu(n)\sum_{c|n,\ c\leq C}\mu(c),
\end{equation}
where $\alpha = \arg z$ and $n = |z|^2$. Recall that by \eqref{taun} and \eqref{coprimePi}, $\beta_z = 0$ if either $\tau(|z|^2)>\tau$ or if $|z|^2$ is not coprime with $\Pi$. We remove the condition $(w\overline{w}, z\overline{z}) = 1$ from \eqref{Bzw} at an acceptable cost as in \cite[(5.10), p.968]{FI1} to get
$$\mathcal{B}'(M, N) = \mathcal{B}'^{\ast}(M, N) + O\left(\left(M^{\frac{1}{4}}N^{\frac{5}{4}}+P^{-1}M^{\frac{3}{4}}N^{\frac{3}{4}}\right)(\log N)^3\right)$$
where
\begin{equation}\label{Bzw2}
\mathcal{B}'(M, N) := \sumsum_{\Imag \overline{w}z \equiv a_0\bmod q_1}\alpha_w\beta_z \zz'(\Real \overline{w}z).
\end{equation}
We then apply Cauchy-Schwarz as in \cite[(5.17), p.970]{FI1} and introduce a smooth radial majorant $f$ supported on the annulus $\frac{1}{2}\sqrt{M}\leq |w|\leq 2\sqrt{M}$ (see \cite[p.970]{FI1}) to get
$$\mathcal{B}'(M, N) \ll M^{\frac{1}{2}}\DD'(M, N)^{\frac{1}{2}},$$
where
$$\DD'(M, N) := \sum_w f(w)\left|\sum_z \beta_z\zz'(\Real \overline{w}z)\right|^2.$$
This eliminates the dependence on $\alpha_w$, so that the sum over $w$ above is free. After inserting a coprimality condition, we arrive at the sum 
\begin{equation}\label{Dast}
\DD'^{\ast}(M, N) := \sumsum_{(z_1, z_2) = 1}\beta_{z_1}\overline{\beta}_{z_2}\mathcal{C}'(z_1, z_2)
\end{equation}
where
$$\mathcal{C}'(z_1, z_2) := \sum_{w}f(w)\zz'(\Real \ow z_1)\zz'(\Real \ow z_2)$$
(see \cite[(5.26), p.972]{FI1} and \cite[(5.27), p.972]{FI1}). The coprimality condition was inserted at the cost 
$$\DD'^{\ast}(M, N) = \DD'(M, N) + O\left(\tau^2(M^{\frac{3}{4}}N^{\frac{3}{4}}+P^{-1}M^{\frac{1}{2}}N^{\frac{3}{2}})(\log MN)^{516}\right)$$
(see \cite[(5.22), p.972]{FI1}). Recall that the congruence condition $c\equiv c_0\bmod q_2$ is hidden in the definition of $\zz'$, while the congruence condition $a\equiv a_0\bmod q_1$ has been removed by restricting the support of $\beta_z$. To prove Lemma \ref{reduct1}, we now have left to prove
\begin{lemma}\label{reduct2}
Let $\eta>0$ and $A>0$, and take $B$ as in \eqref{bigB}. Then there exists $x_0 = x_0(\eta, A)$ such that for all $x\geq x_0$, for all $M$ and $N$ satisfying \eqref{condN} and \eqref{condMN}, and for all $C$ satisfying \eqref{eqC}, we have
\begin{equation}\tag{B''}
\left|\mathcal{D}'^{\ast}(M, N)\right|\leq C \vartheta^2\theta^4M^{\frac{1}{2}}N^{\frac{3}{2}}(\log MN)^{10}.
\end{equation}
\end{lemma}
Note the extra factor of $\theta$ coming from the restriction of support of $\beta$ to a sector of angle $\theta$. 

\subsection{Proof of Lemma \ref{reduct2}}In order to obtain this upper bound, Friedlander and Iwaniec introduce a quantity they call the ``modulus''
$$\Delta = \Delta(z_1, z_2) = \Imag(\oz_1 z_2),$$
which is non-zero whenever $(z_1, z_2) = 1$ and $z_1$ and $z_2$ are odd and primitive. The sum defining $\mathcal{D}'^{\ast}(M, N)$ is split into several different sums depending on the size of the modulus $\Delta$. Different techniques are used to treat each of these sums, but we will manage to avoid going into the details by reducing our sums to those already studied in \cite{FI1}.
\\\\
The Fourier analysis carried out on \cite[p.974]{FI1} depends on the greatest common denominator of $\Delta$ and $q_2$. Using the Poisson summation formula similarly as on \cite[p.974]{FI1}, equation \eqref{Dast} can now be written as
$$\mathcal{D}'^{\ast}(M, N) = \sum_{\delta|q_2}\sumsum_{\substack{(z_1, z_2) = 1\\ (q_2, |\Delta|) = \delta}}\beta_{z_1}\overline{\beta}_{z_2}\mathcal{C}'(z_1, z_2),$$
where 
\begin{equation}\label{cprime}
\mathcal{C}'(z_1, z_2) = (q_2/\delta)^{-2}|z_1z_2|^{-1/2}\sum_{h_1}\sum_{h_2}F\left(\frac{h_1}{|\Delta z_2|^{1/2}q_2/\delta}, \frac{h_2}{|\Delta z_1|^{1/2}q_2/\delta}\right)G'(h_1, h_2);
\end{equation}
the Fourier integral
$$F(u_1, u_2) = \int\int f\left(\frac{z_2}{|z_2|}t_1^2-\frac{z_1}{|z_1|}t_2^2\right)e(u_1 t_1+u_2 t_2)dt_1 dt_2$$
is the same as the one defined in \cite[(6.8), p.974]{FI1} and
$$G'(h_1, h_2) = \frac{1}{|\Delta|}\sumsum_{\substack{\gamma_1, \gamma_2\bmod|\Delta| \\ \gamma_1^2z_2\equiv \gamma_2^2z_1 \bmod{|\Delta|} \\ \gamma_1 \equiv \gamma_2 \equiv c_0 \bmod \delta}}e\left(\frac{\gamma'_1h_1+\gamma'_2h_2}{|\Delta|q_2/\delta}\right)$$
is an arithmetic sum similar to $G(h_1, h_2)$ defined in \cite[(6.10), p.974]{FI1}, but now incorporating the congruence condition $c\equiv c_0\bmod q_2$; here $\gamma'_i$ is the solution (modulo $\frac{|\Delta| q_2}{\delta}$) to the system of congruences 
$$
\begin{cases}
	\gamma'_i \equiv \gamma_i\bmod |\Delta| \\
	\gamma'_i\equiv c_0\bmod q_2.
\end{cases}
$$
Such a solution is guaranteed to exist because $\gamma_1 \equiv \gamma_2 \equiv c_0 \bmod \delta$. Note that similarly as in \cite{FI1}, we omit in the notation the dependence of $F$ and $G'$ on $z_1$ and $z_2$. 
\\\\
The \textit{main term} in the above expansion for $\mathcal{C}'(z_1, z_2)$ comes, as usual, from the terms with $h_1 = h_2 = 0$ in equation \eqref{cprime}. Similarly as in the proof of condition (R) above, we don't need to make any changes in the treatment of the Fourier integral; \cite[Lemma 7.1, p.976]{FI1} and \cite[Lemma 7.2, p.977]{FI1} are still valid, with the implied constants now depending on $q_2$ as well. We recall that \cite[Lemma 7.2, p.977]{FI1} states that for $z_1$ and $z_2$ in the support of $\beta_z$ we have
\begin{equation}\label{F0}
F_0(z_1, z_2):= F(0, 0) = 2\hat{f}(0)\log 2|z_1z_2/\Delta| + O(\Delta^2M^{\frac{1}{2}}N^{-2}\log N).
\end{equation}
We now have to give an upper bound for $G'(h_1, h_2)$ similar to the bound given in \cite[Lemma 8.1, p.978]{FI1}, as well as give an exact formula for 
$$G'_0(z_1, z_2) := G'(0, 0)$$
similar to the one in \cite[Lemma 8.4, p.980]{FI1}. This is where we now specialize to the case 
$$q_2 = 4 \text{ and } c_0\in\{0, 2\}.$$
Recall that we restricted the support of $\beta_z$ to $z$ in a fixed congruence class modulo $64q_1$. Hence $z_1\equiv z_2 \bmod 64$, so that $\Delta = \Imag(\oz_1 z_2) \equiv 0\bmod 64$. This significantly simplifies our arguments since now $\delta = (4, |\Delta|) = 4$.
\\\\
The arithmetic sum $G'(h_1, h_2)$ now simplifies to
$$G'(h_1, h_2) = \frac{1}{|\Delta|}\sumsum_{\substack{\gamma_1, \gamma_2\bmod |\Delta| \\ \gamma_1^2z_2\equiv \gamma_2^2z_1 \bmod{|\Delta|} \\ \gamma_1 \equiv \gamma_2 \equiv c_0 \bmod 4}}e\left(\frac{\gamma_1h_1+\gamma_2h_2}{|\Delta|}\right).$$
We first prove a lemma analogous to \cite[Lemma 8.1, p.978]{FI1}. 
\begin{lemma}\label{lem81}
Fix $\theta \in \{2, 4\}$ and let
$$G''(h_1, h_2; \theta) = \frac{1}{|\Delta|}\sumsum_{\substack{\gamma_1, \gamma_2\bmod |\Delta| \\ \gamma_1^2z_2\equiv \gamma_2^2z_1 \bmod{|\Delta|} \\ \gamma_1 \equiv \gamma_2 \equiv 0 \bmod \theta}}e\left(\frac{\gamma_1h_1+\gamma_2h_2}{|\Delta|}\right).$$
Then 
\begin{equation}\label{res1}
\left|G''(h_1, h_2; \theta) \right|\leq 16\tau_3(\Delta)|\Delta|^{-1}(z_1h_1^2-z_2h_2^2, \Delta).
\end{equation}
\end{lemma}
Introducing a change of variables $\gamma_1 = \theta\omega_1$ and $\gamma_2 = \theta\omega_2$, we get
$$G''(h_1, h_2; \theta) = \frac{1}{|\Delta|}\sumsum_{\substack{\omega_1, \omega_2\bmod |\Delta|/\theta \\ \omega_1^2z_2\equiv \omega_2^2z_1 \bmod{|\Delta|/\theta^2}}}e\left(\frac{\omega_1h_1+\omega_2h_2}{|\Delta|/\theta}\right).$$
Proceeding in a similar fashion as on \cite[p.977-978]{FI1}, we write
$$\Delta/\theta = \theta\Delta_1(\Delta_2)^2,$$
with $\Delta_1$ squarefree. The condition $\omega_1^2z_2\equiv \omega_2^2z_1 \bmod{|\Delta|/\theta^2}$ implies that $(\omega_1^2, \Delta/\theta^2) = (\omega_2^2, \Delta/\theta^2)$, so we can write 
$$(\omega_1^2, \Delta/\theta^2) = (\omega_2^2, \Delta/\theta^2) = d_1d_2^2$$
with $d_1$ squarefree. Then $d_1|\Delta_1$, $d_2|\Delta_2$, $(d_1, \Delta_2/d_2) = 1$, and we can make a change of variables $\omega_i = d_1d_2\eta_i$, there $\eta_i$ runs over the residue classes modulo $|\Delta|/\theta d_1d_2$ and coprime with $|\Delta|/\theta^2 d_1d_2^2$. Setting $b_1 = \Delta_1/d_1$ and $b_2 = \Delta_2/d_2$, the analogue of the equation on top of \cite[p.978]{FI1} becomes
$$G''(h_1, h_2; \theta) = \frac{1}{|\Delta|} \sumsum_{\substack{b_1d_1 = |\Delta_1| \\ b_2d_2 = \Delta_2 \\ (d_1, b_2) = 1}}\sumsum_{\substack{\eta_1, \eta_2\bmod \theta b_1b_2^2d_2 \\ (\eta_1\eta_2, b_1b_2) = 1 \\ \eta_1^2z_2\equiv \eta_2^2z_1\bmod b_1b_2^2}}e((\eta_1h_1+\eta_2h_2)/\theta b_1b_2^2d_2)$$
The innermost sum vanishes unless $h_1\equiv h_2\equiv 0\bmod \theta d_2$, so $G''(h_1, h_2)$ is equal to 
$$\frac{1}{|\Delta|}\sum_{\substack{b_1d_1 = |\Delta_1|\\ (d_1, b_2) = 1}}\sum_{\substack{b_2d_2 = \Delta_2\\ \theta d_2|(h_1, h_2)}}\theta^2d_2^2\sumsum_{\substack{\eta_1, \eta_2\bmod b_1b_2^2 \\ (\eta_1\eta_2, b_1b_2) = 1 \\ \eta_1^2z_2\equiv \eta_2^2z_1\bmod b_1b_2^2}}e((\eta_1h_1+\eta_2h_2)/\theta b_1b_2^2d_2).$$ 
Performing the change of variables $\eta_2 = \omega\eta_1$, the analogue of equation \cite[(8.3), p.978]{FI1} becomes 
$$\frac{1}{|\Delta|}\sum_{\substack{b_1d_1 = |\Delta_1|\\ (d_1, b_2) = 1}}\sum_{\substack{b_2d_2 = \Delta_2\\ \theta d_2|(h_1, h_2)}}\theta^2d_2^2\sum_{\omega\equiv z_2/z_1\bmod b_1b_2^2}R((h_1+\omega h_2)(\theta d_2)^{-1}; b_1b_2^2),$$
where $R(h; b)$ is the classical Ramanujan sum defined on \cite[p.978]{FI1}. Now the same argument as on \cite[p.978]{FI1} yields the desired upper bound \eqref{res1}.$\Box$ 
\\\\
We now turn our attention back to $G'(h_1, h_2)$. In case $c_0 = 0$, we're in the case of Lemma \ref{lem81} and
$$\left|G'(h_1, h_2)\right| = \left|G''(h_1, h_2; 4)\right| \leq 16\tau_3(\Delta)|\Delta|^{-1}(z_1h_1^2-z_2h_2^2, \Delta).$$
If, on the other hand, $c_0 = 2$, we note that $G'(h_1, h_2) = G''(h_1, h_2; 2) - G''(h_1, h_2; 4)$ since $\Delta\equiv 0\bmod 16$. Hence
$$\left|G'(h_1, h_2)\right|\leq 32\tau_3(\Delta)|\Delta|^{-1}(z_1h_1^2-z_2h_2^2, \Delta).$$
The same arguments as those in Section 9 of \cite{FI1} now suffice to show that the main term in the Fourier expansion indeed comes from $h_1 = h_2 = 0$. We recall the result from \cite[(9.10), p.983]{FI1} here. Let 
$$\mathcal{D}'_0(M, N) := \sumsum_{\substack{(z_1, z_2) = 1}}\beta_{z_1}\overline{\beta}_{z_2}\mathcal{C}'_0(z_1, z_2),$$
where 
\begin{equation}\label{czeroprime}
\mathcal{C}'_0(z_1, z_2) = |z_1z_2|^{-1/2}F_0(z_1, z_2)G'_0(z_1, z_2).
\end{equation}
\begin{lemma}\label{prop910}
Let $\eta>0$ and $A>0$, and take $B$ as in \eqref{bigB}. Then there exists $x_0 = x_0(\eta, A)$ such that for all $x\geq x_0$, for all $M$ and $N$ satisfying \eqref{condN} and \eqref{condMN}, and for all $C$ satisfying \eqref{eqC}, we have 
$$\left|\DD'^{\ast}(M, N) - \DD'_0(M, N)\right| \leq \vartheta^{-1}\tau^2N^2(\log N)^{\eta^{-1/\eta}},$$
where $\tau$ is defined in \eqref{taun}. 
\end{lemma}
It now remains to estimate $\DD'_0(M, N)$. We turn to obtaining an exact formula for $G'_0(z_1, z_2)$. Recall, from top of \cite[p.979]{FI1}, that 
$$G_0(z_1, z_2) := \frac{1}{|\Delta|}\sumsum_{\substack{\gamma_1, \gamma_2\bmod|\Delta| \\ \gamma_1^2z_2\equiv \gamma_2^2z_1 \bmod{|\Delta|}}}1 = N(z_2/z_1; |\Delta|)/|\Delta|,$$
where $N(a; r)$ denotes the number of solutions $(\gamma_1, \gamma_2)$ modulo $r$ to 
$$a\gamma_1^2\equiv \gamma_2^2\bmod r.$$
Similarly, 
$$G'_0(z_1, z_2)= N'(z_2/z_1; |\Delta|)/|\Delta|,$$
where $N'(a; r)$ is the number of solutions $(\gamma_1, \gamma_2)$ modulo $r$ to the congruences
$$
\begin{cases} 
a\gamma_1^2\equiv \gamma_2^2\bmod r\\
\gamma_1\equiv\gamma_2\equiv c_0\bmod 4.
\end{cases}
$$
Since $z_2/z_1\equiv 1\bmod 64$ and $\Delta\equiv 0\bmod 64$, we are only concerned with the case $a\equiv 1\bmod 64$ and $r\equiv 0\bmod 64$. 

\subsection{Computation of $N'(a; r)/r$}

\subsubsection{Case $c_0 = 0$}
First let us compute $N'(a; r)/r$ when $c_0 = 0$. Since $\gamma_1\equiv\gamma_2\equiv 0\bmod 4$, we can make a change of variables $\gamma_1 = 4\omega_1$ and $\gamma_2 = 4\omega_2$, where now $\omega_i$ are congruence classes modulo $r/4$, to find that $N'(a; r) = 16N(a; r/16)$, i.e. 
$$N'(a; r)/r = N(a; r/16)/(r/16).$$
This leads to a formula of type \cite[(8.16), p.980]{FI1}. If $16\cdot 2^{\nu}$ with $\nu\geq 1$ is the exact power of $2$ dividing $\Delta$, we get
$$G'_0(z_1, z_2) = \nu\sum_{\substack{16d|\Delta \\ d\text{ odd}}}\frac{\varphi(d)}{d}\left(\frac{z_2/z_1}{d}\right).$$
Since $\Delta\equiv 0\bmod 64$, we are only interested in the case $\nu \geq 2$, where this becomes
\begin{equation}\label{Gprime00}
G'_0(z_1, z_2) = 2\sum_{64d|\Delta}\frac{\varphi(d)}{d}\left(\frac{z_2/z_1}{d}\right),
\end{equation}
by the same reasoning as in \cite[Lemma 8.4, p.980]{FI1}. 

\subsubsection{Case $c_0 = 2$}
When $c_0 = 2$ and $4|r$, we can make a change of variables $\gamma_1 = 2\omega_1$ and $\gamma_2 = 2\omega_2$ so that $N'(a; r)$ is $4$ times the number of solutions $(\omega_1, \omega_2)$ modulo $r/4$ to the system of congruences 
$$
\begin{cases}
\omega_1\equiv \omega_2\equiv 1\bmod 2 \\
a\omega_1^2\equiv \omega_2^2\bmod r/4.
\end{cases}
$$
When $16|r$, we must subtract from $4N(a; r/4)$ those solutions with $\omega_1\equiv \omega_2\equiv 0\bmod 2$. This gives $N'(a; r) = 4N(a; r/4) - 16N(a; r/16)$, i.e.
$$\frac{N'(a; r)}{r} = \frac{N(a; r/4)}{r/4} - \frac{N(a; r/16)}{r/16}.$$
Hence if $16\cdot 2^{\nu}$ with $\nu\geq 2$ is the exact power of $2$ dividing $\Delta$, we get
\begin{equation}\label{Gprime02}
G'_0(z_1, z_2) = 2\sum_{16d|\Delta}\frac{\varphi(d)}{d}\left(\frac{z_2/z_1}{d}\right) - 2\sum_{64d|\Delta}\frac{\varphi(d)}{d}\left(\frac{z_2/z_1}{d}\right),
\end{equation}
which is the analogue of \eqref{Gprime00}. 
\subsection{End of proof of of Lemma \ref{reduct2}}
We now turn back to estimating $\DD'_0(M, N)$. As in \cite[(10.4), p.985]{FI1}, we can use \eqref{F0} to write
$$\DD'_0(M, N) = 2\hat{f}(0)N^{\frac{1}{2}}T'(\beta)+O\left((\tau^{-1}+\theta)Y'(\beta)M^{\frac{1}{2}}N^{-\frac{1}{2}}\log N\right)$$
where
$$T'(\beta) := \sumsum_{(z_1, z_2) = 1}\beta_{z_1}\overline{\beta}_{z_2}G'_0(z_1, z_2)\log 2|z_1z_2/\Delta|$$
and
$$Y'(\beta) := \sumsum_{(z_1, z_2) = 1}|\beta_{z_1}\overline{\beta}_{z_2}|\tau(|z_1|^2)\tau(|z_2|^2)\tau_3(\Delta).$$
Similarly as in \cite[Lemma 10.1, p.985]{FI1}, we can bound $Y'(\beta)$ by
$$Y'(\beta)\ll \theta^4N^2(\log N)^{2^{19}},$$
so that we are left with estimating the sum  $T'(\beta)$. In each of the cases $c_0 = 0$ and $c_0 = 2$, we can use the formula for $G'_0(z_1, z_2)$ and $F_0(z_1, z_2)$ to write $T'(\beta)$ as a sum similar to \cite[(10.13), p.986]{FI1}. If we define
$$T'(\beta, \xi) := 2\sum_{d}\frac{\varphi(d)}{d}\sumsum_{\substack{(z_1, z_2) = 1\\ \Delta(z_1, z_2)\equiv 0\bmod \xi d}}\beta_{z_1}\overline{\beta}_{z_2}\left(\frac{z_2/z_1}{d}\right)\log 2|z_1z_2/\Delta|,$$
then
$$T'(\beta) = 
\begin{cases}
T'(\beta, 64) & \text{ if }c_0 = 0 \\
T'(\beta, 16)-T'(\beta, 64) & \text{ if }c_0 = 2
\end{cases}
$$
Lemma \ref{reduct2} now follows from this analogue of \cite[Proposition 10.2, p.986]{FI1}:
\begin{lemma}\label{reduct3}
Fix $\xi \in\{16, 64\}$. Let $\eta>0$, $A>0$, and $\sigma>0$, and take $B$ as in \eqref{bigB}. Then there exists $x_0 = x_0(\eta, A)$ and $C_0 = C_0(\eta, A, \sigma)>0$ such that for all $x\geq x_0$, for all $N$ satisfying \eqref{condN}, and for all $C$ satisfying \eqref{eqC}, we have
$$T'(\beta, \xi) \leq C_0 N^2(\log N)^{-\sigma}+P^{-1}N^2\log N,$$
where $P$ is any number in the range \eqref{Prange}. 
\end{lemma}
We recall that $N$ and $P$ appear as parameters restricting the support of $\beta_z$; see \eqref{betaz}.

\subsection{Proof of Lemma \ref{reduct3}: oscillations of characters and symbols}
Although complicated, the proof of \cite[Proposition 10.2]{FI1} generalizes directly to the proof of Lemma \ref{reduct3}. One can check in \cite[Sections 15-17]{FI1} that the same arguments are valid when $\xi = 16$ or $64$ instead of $\xi = 4$. For instance, on \cite[p.1005]{FI1} and \cite[p.1015]{FI1}, one now sums over multiplicative characters of the groups $(\ZZ[i]/\xi d\ZZ[i])^{\times}$ and $(\ZZ[i]/\xi bd\ZZ[i])^{\times}$, respectively. 
\\\\
Moreover, the restriction on the support of $\beta_z$ to $z$ in a fixed primary congruence class modulo $64q_1$ (where $q_1$ is as in \eqref{anprime}) as opposed to modulo $8$ is handled in the same way as in \cite[Sections 15-17]{FI1}. For sums over medium-size moduli, the estimation of $\beta_z$ is trivial and so the restriction on the support is irrelevant (see bottom of \cite[p.1003]{FI1}). For sums over small moduli, i.e. $d$ of size at most a large power of $\log N$, the key sum to bound from above is the character sum
\begin{equation}\label{SKB}
S_{\chi}^k(\beta) = \sum_{z}\beta_z\chi(z)\left(\frac{z}{|z|}\right)^k,
\end{equation}
where $\chi$ is a multiplicative character of the group $(\ZZ[i]/\xi d\ZZ[i])^{\times}$ (see \cite[(16.14), p.1005]{FI1}). The restriction on the support of $\beta_z$ can be detected by multiplicative characters modulo $64q_1$, so that we can simply transform $\chi$ into a character for the group $(\ZZ[i]/64 q_1 d\ZZ[i])^{\times}$. The sum \eqref{SKB} is bounded by studying the Hecke $L$-functions
$$L(s, \psi) = \sum_{\aaa}\psi(\aaa)(N\aaa)^{-s},$$
where the sum ranges over the non-zero odd ideals $\aaa$ of $\ZZ[i]$ and 
$$\psi(\aaa):=\chi(z)\left(\frac{z}{|z|}\right)^k$$
where $z$ is the unique primary Gaussian integer which generates $\aaa$. The dependence on $\chi$ of the bound given for $S_{\chi}^k(\beta)$ is only through the modulus of $\chi$ (see \cite[Lemma 16.2, p.1012]{FI1}) and this modulus is different from $4d$ by a fixed constant. Similarly, for the sums over large moduli, the  key sum to bound from above is the character sum
\begin{equation}\label{SKBP}
S_{\chi}^k(\beta') = \sum_{z}\beta'_z\chi(z)\left(\frac{z}{|z|}\right)^k,
\end{equation}
where $\chi$ is a multiplicative character of the group $(\ZZ[i]/\xi bd\ZZ[i])^{\times}$ (where $b$ is an integer and $d$ is again bounded by a large power of $\log N$) but $\beta'_z$ is now 
$$\beta'_z = i^{\frac{r-1}{2}}\left(\frac{s}{|r|}\right)\beta_z$$
if $z = r+is$ (see \cite[(17.8), p.1014]{FI1} and \cite[(17.12), p.1015]{FI1}). Again, the restriction on the support of $\beta_z$ (and hence also $\beta'_z$) can be detected by multiplicative characters modulo $64q_1$, so that we can transform $\chi$ into a character for the group $(\ZZ[i]/64q_1 bd\ZZ[i])^{\times}$. Cancellation in the sum \eqref{SKBP} is now achieved due to the oscillation of the symbol
$$i^{\frac{r-1}{2}}\left(\frac{s}{|r|}\right)$$
as $z$ varies over primary Gaussian integers, but again the dependence on $\chi$ of the bound given for \eqref{SKBP} is only through the modulus of $\chi$ (see \cite[Proposition 17.2, p.1016]{FI1}) and this modulus is again different from $4bd$ by a fixed constant. This shows that Lemma \ref{reduct3} follows from \cite[Proposition 10.2]{FI1} and hence Proposition \ref{prop41} is proved.

\section{Acknowledgements}
I would like to give special thanks to my advisors \a'{E}tienne Fouvry and Peter Stevenhagen for their unceasing support and useful advice as well as for helping me resolve numerous issues that arose during the course of this research. I would also like to thank Christian Elscholtz, Jan-Hendrik Evertse, Florent Jouve, and Hendrik Lenstra for useful discussions.    

\bibliographystyle{plain}
\bibliography{sixteenrankreferences}
\end{document}